\documentclass[a4paper,onecolumn,10pt]{article}

\usepackage{amsfonts, epsfig, amsmath, amssymb,amsthm, color}
\usepackage{mathrsfs}

\usepackage[ngerman, english]{babel}

\let\oldmarginpar\marginpar
\renewcommand{\marginpar}[1]{\oldmarginpar{\scriptsize\texttt{\color{blue}{#1}}}}

\usepackage{float}
\usepackage{caption}
\usepackage{subcaption}
\usepackage{mathtools}
\usepackage{amsfonts,amsmath, amssymb,amsthm}
\usepackage{mathrsfs}
\usepackage{mathtools}
\usepackage{ marvosym }
\usepackage{multicol}
\usepackage{listings} 
\usepackage{lipsum}
\usepackage{hyperref}
\usepackage{enumitem}
\usepackage{tikz}

\usetikzlibrary{fit,matrix}

\DeclareMathOperator{\Var}{Var}
\DeclareMathOperator{\Cov}{Cov}

\DeclareMathOperator{\sgn}{sgn}

\newtheorem{thm}{Theorem}[section]

\newtheorem{lem}[thm]{Lemma}

\newtheorem{rem}[thm]{Remark}

\newcommand{\be}{\begin{equation}}
\newcommand{\ee}{\end{equation}}
\newcommand{\bes}{\begin{equation*}}
\newcommand{\ees}{\end{equation*}}

\newcommand{\ba}{\begin{equation}\begin{aligned}}
\newcommand{\baa}{\begin{equation*}\begin{aligned}}
\newcommand{\ea}{\end{aligned}\end{equation}}

\newcommand{\ban}{\begin{equation*}\begin{aligned}}
\newcommand{\ean}{\end{aligned}\end{equation*}}

\newcommand{\abs}[1]{\left\vert#1\right\vert}

\newcommand{\vecl}{\left(\begin{array}{c}}
\newcommand{\vecr}{\end{array}\right)}

\renewcommand{\d}{\text{d}}

\newcommand{\E}{\mathbf{E}}
\renewcommand{\P}{\mathbf{P}}

\newcommand{\ex}{\mathrm{e}}
\newcommand{\di}{\mathrm{d}}

\newcommand{\rB}{\mathscr{B}}

\newcommand{\rF}{\mathscr{F}}

\newcommand{\e}{\varepsilon}

\newcommand{\bI}{\mathbb{I}}

\newcommand{\bR}{\mathbb{R}}

\newcommand{\bF}{\mathbb{F}}

\begin{document}
\title{Stochastic Energy-Balance Model With A Moving Ice Line}


\date{\today}

\author{Ilya Pavlyukevich\footnote{Institute of Mathematics, Friedrich Schiller University Jena, Ernst--Abbe--Platz 2,
07743 Jena, Germany; ilya.pavlyukevich@uni-jena.de} \ and Marian Ritsch\footnote{Institute of Mathematics, Friedrich Schiller University Jena, Ernst--Abbe--Platz 2,
07743 Jena, Germany; carl.christian.marian.ritsch@uni-jena.de}}

\maketitle

\begin{abstract}
In [SIAM J.\ Appl.\ Dyn.\ Sys., 12(4):2068--2092, 2013], Widiasih proposed and analyzed a deterministic one-dimensional 
Budyko--Sellers energy-balance model with a moving ice-line. In this paper, we extend this model to the stochastic 
setting and analyze it within the framework of stochastic slow-fast systems.
We derive the dynamics for the ice line in the limit of a small parameter as a solution to a stochastic differential equation. The stochastic approach enables the study of co-existing (metastable) climate states as well as the transition dynamics between them.
\end{abstract}


\tableofcontents
\smallskip


\numberwithin{equation}{section}

\section{Introduction and motivation}

In this paper, we investigate the behavior of a stochastic version of an energy-balance model (EBM) with a moving ice line. This model was initially introduced and studied in a series of papers by Widiasih and co-authors, as documented in \cite{widiasih2013dynamics,walsh2014dynamics}. 

Energy-balance climate models serve as conceptual, low-dimensional representations of Earth's global climate, 
relying on simplified modeling of the planetary energy budget.
The paradigm of EBMs traces back to the seminal $0$-dimensional models developed by Budyko \cite{Budyko-69} and Sellers \cite{Sellers-69}. 
These models are expressed as follows:
\ba
\label{e:EBM_gen}
\frac{\di}{\di t} X(t)=\frac{1}{R}\Big( R_\text{input}(X(t))  - R_\text{output}(X(t)) \Big),
\ea
where $X(\cdot)$ represents the temperature of the Earth-atmosphere system and $R$ denotes Earth's heat capacity.
The absorbed incoming solar radiation is parameterized as
\ba
R_\text{input}(X)=Q(1-\alpha(X)),    
\ea
where $Q$ signifies the global mean solar radiative input and $\alpha=\alpha(\cdot)$ denotes Earth's albedo representing
the proportion of reflected incoming radiation. 

In the 0-dimensional model, it is assumed that the albedo depends on Earth's temperature, with $\alpha=\alpha(X)$. Specifically, the albedo is high for low temperatures ($\alpha=\alpha_\text{snow}=0.62$) and low for high temperatures ($\alpha=\alpha_\text{water}=0.32$) due to the contrasting albedos of ice or snow cover versus water or land.

The outgoing radiation is commonly parameterized using the modified Stefan--Boltzmann law:
\ba
\label{e:SB}
 R_\text{output}(X)= g(X)X^4,
\ea
where $g(\cdot )$ represents the grayness of the system.
 
In a 0-dimensional setting, 
all quantities are averaged over the Earth's surface, 
neglecting variations between tropical and polar regions. It is essential to note that, in reality, 
the absence of ice in tropical latitudes leads to greater absorption of solar radiation and enhanced cooling to space via infrared radiation. As a result, the net radiation balance is positive in tropical regions and negative in polar regions. To maintain equilibrium in the climate system, it becomes necessary to introduce a heat transport term into the model that redistributes heat poleward.

More realistic $1$-dimensional EBMs were first introduced by 
Schneider and Gal-Chen (1973) \cite{schneider1973numerical}, Held and Suarez \cite{held1974simple}, North \cite{north1975analytical}, Ghil \cite{ghil1976climate}. The unknown variable in these models is the 
temperature profile $X=X(t,x)$, where $t\geq 0$ is the time, and $x \in [0, 1]$ is the longitude proxy, with $x=0$ corresponding to the Equator and $x=1$ corresponding to the North Pole.  The energy balance equation \eqref{e:EBM_gen} takes the form
\ba
\label{e:EBM1}
\frac{\partial X(t,x)}{\partial t}=\frac{1}{R}\Big(  Q(x)(1-\alpha(x,\eta)) - (a+bX(t,x)) -c\Big( X(t,x) -\int_0^1 X(t,z)\,\di z   \Big)\Big).  
\ea
In this equation, $Q=Q(x)$ is the mean incoming solar radiation at the top of the atmosphere on the longitude $x$.
The albedo $\alpha=\alpha(x,\eta)$ depends on the parameter $\eta\in[0,1]$ that represents the latitude of the ice line --- the border between 
the ice-covered and ice-free areas.
It is assumed that the albedo
is high for latitudes $x$ poleward of the ice-line, i.e., for $x\in(\eta,1]$ and low for lower latitudes, $x\in [0,\eta)$.
Typically the dependence is assumed to be of the form
\ba
\label{e:albedo}
\alpha(x,\eta)&= \frac{\alpha_\text{snow}+\alpha_\text{water}}{2}  + \frac{\alpha_\text{snow}-\alpha_\text{water}}{2} \tanh(K(x-\eta)),\\
\ea
with some large constant $K\gg 1$.

The term $a+b X$ is the linearized form of the Stefan--Boltzmann law \eqref{e:SB}. 

The transport term, $c\big( X(t,x) -\int_0^1 X(t,z)\,\di z \big)$, eventually parameterizes the heat transfer between latitudes. 
It is important to note that, contrary to diffusive transport usually depicted by the local diffusion operator, the heat transfer process is a global phenomenon that takes into account the mean Earth's temperature $\int_0^1 X(t,z)\,\di z$.

In climatological models, the constants $R$, $a$, $b$ and $c$ are chosen appropriately, see Table~1 in Widiasih \cite{widiasih2013dynamics} 
and Section \ref{s:example} for specific values. 

Both $0$- and $1$-dimensional models exhibit two stable solutions, $X_\text{cold}$ and $X_\text{warm}$, roughly corresponding to 
cold/snowball/glacial and warm/inter-glacial climate conditions. 
However, in these deterministic models, transitions between climate states, or co-existing meta-stable climate states, are not possible.

Therefore, more realistic climate models must consider fast atmospheric fluctuations, which are incorporated into the equation through 
stochastic terms (white or red noise), as proposed by Hasselmann \cite{Hasselmann76}. 
The addition of noise terms significantly alters the qualitative behaviour of the model, as even very small stochastic perturbations 
induce transitions of a system between stable states, resulting in a stochastic climate model demonstrating the phenomenon of \emph{metastability}.
For instance, Benzi et al. \cite{BenziPSV-81,BenziPSV-82,BenziPSV-83} and Nicolis and Nicolis \cite{nicolis1981stochastic,nicolis1981solar,Nicolis-82} conjectured the \emph{stochastic resonance} connection between glaciation cycles and astronomical Milankovitch cycles in a noisy 0-dimensional EMB.

Further information on deterministic and stochastic EMBs can be found in works by Ghil \cite{ghil1981energy}, North et al. \cite{north1981energy}, Imkeller \cite{Imkeller-01}, Dijkstra \cite{dijkstra2013nonlinear}, and Ghil and Lucarini \cite{ghil2020physics}.

Our present work is motivated by a series of papers by Widiasih, Walsh, and Rackauckas \cite{widiasih2013dynamics,walsh2014dynamics,walsh2015budyko} dedicated to the dynamical 1-dimensional EBM with a \emph{moving} ice line. In these papers, it was assumed that the ice line coordinate $\eta$ in \eqref{e:EBM1} and \eqref{e:albedo}, treated as a fixed parameter, evolves dynamically with time according to the ordinary differential equation (ODE)
\ba
\label{e:EBM-eta}
\frac{\di \eta(t)}{\di t}= \e\Big(X(t,\eta(t))-X_\text{critical}\Big),
\ea
where $X_\text{critical}=-10^\circ$C is the critical temperature, and $0<\varepsilon \ll 1$ (with $\varepsilon=10^{-2}$ in Widiasih \cite{widiasih2013dynamics}) is a small time scale parameter accounting for the slow dynamics of the ice line compared to the relatively fast fluctuations of the temperature profile $t\mapsto X(t,\cdot)$.

In \cite{widiasih2013dynamics,walsh2014dynamics,walsh2015budyko}, the authors extensively studied the systems of equations \eqref{e:EBM1} and \eqref{e:EBM-eta} for the temperature profile $X=X(t,x)$ and the ice line $\eta=\eta(t)$, discovering stable temperature profiles and ice line equilibria for various parameterizations of the albedo function.

In this paper, we investigate the \emph{stochastic} EBM obtained by perturbing the system \eqref{e:EBM1}, \eqref{e:EBM-eta} with infinitely dimensional multiplicative noise. We determine the limit dynamics of the ice line $\eta$ as a solution of a stochastic differential equation. In the example situation, the limit system will have two stable coexisting equilibria corresponding to warm and cold climatic conditions.

From a mathematical perspective, our results are derived within the framework of non-linear slow-fast systems and stochastic averaging,  
see, e.g., the classical 
works by Skorokhod \cite{Skorokhod-89} and Freidlin and Wentzel \cite{FreidlinW-12}, as well as a recent paper by Assing et al.\ \cite{assing2021stochastic}. One of the main challenges lies in the proper treatment of the composition term $X(t,\eta(t))$, where both the temperature profile $X(t,\cdot)$ and the ice line $\eta(t)$ are unknown random processes. To obtain a well-defined solution $(X(\cdot,\cdot),\eta(\cdot) )$, we demonstrate, under certain assumptions, that the random mappings $x\mapsto X(t,\cdot)$ are $C^2$ functions with norms $\|X(t,\cdot)\|_\infty$ and $\|\nabla X(t,\cdot)\|_\infty$ bounded in a certain 
sense for all $t\geq 0$. These stability results are obtained with the help of Sobolev embedding inequalities.

The paper is organized as follows. In Section \ref{s:setting}, we formulate the set of assumptions and present the main results. 
Section \ref{s:example} is dedicated to the analysis of our stochastic version of Widiasih's model. 
In Section \ref{s:existence}, we establish the existence and uniqueness of a solution $(X,\eta)$. 
In Section \ref{s:01}, we demonstrate, under certain assumptions, that the ice line does not hit the boundaries $\eta=0$ and $\eta=1$, allowing us to consider non-degenerate perpetual dynamics confined to the interval $(0,1)$. 
The stability of the solution $t\mapsto X(x,\cdot)$ for $t\geq 0$ is established in Section \ref{s:stability}. Finally, the convergence of the ice line dynamics as $\varepsilon\to 0$ in the weak and strong sense is obtained in Section \ref{s:convergence}.
 
\medskip
\noindent
\textbf{Notation.} For $X=X(t,x)$, we denote $\nabla X(t,x):=\frac{\partial}{\partial x}X(t,x)$ and $\nabla^2 X(t,x):=\frac{\partial^2}{\partial x^2}X(t,x)$. 
For a function $F=F(x,\eta, X,Z)$, its partial derivatives are denoted by $F_x$, $F_X$, $F_{xx}$ etc.
For $\eta\in\bR$, we denote by $\eta|^1_0$ its truncated value
\ba
\eta|^1_0:=\begin{cases}
         1,\eta>1,\\
         \eta, \eta\in [0,1],\\
         0,\eta<0.
        \end{cases}
\ea
As usual, $C^k([0,1],\bR)$ is the space of $k$-times ($k\in\mathbb N_0$) continuously differentiable functions $f\colon [0,1]\mapsto \bR$.
The $\sup$-norm of $f$ is denoted by 
\ba
\|f\|_\infty=\sup_{x\in[0,1]}|f(x)|.
\ea
By $\text{Lip}(f)$ we denote a Lipschitz constant of a Lipschitz continuous function $f$.
The Lebesgue space
$L^p([0,1],\bR)$, $p\geq 1$,
is the space of (equivalence classes of) measurable functions $f\colon [0,1]\mapsto \bR$ with the finite norm 
\ba
\|f\|_{L^p}^p=\int_0^1 |f(x)|^p\,\di x<\infty.
\ea
For $p\geq 1$, the Sobolev space $W^{1,p}([0,1],\bR)$ consists of 
absolutely continuous functions on $[0,1]$ whose derivatives $\nabla f$ belong to $L^p([0,1],\bR)$. 
Analogously, the Sobolev space $W^{2,p}([0,1],\bR)$
consists of smooth functions on $[0,1]$ whose derivatives $\nabla f$ 
are absolutely continuous on $[0,1]$ and $\nabla^2 f$
belongs to $L^p([0,1],\bR)$. 

The Sobolev space norm of $f \in W^{k,p}([0,1],\bR)$, $k=1,2$, is defined as
\ba
\|f\|_{ W^{k,p}}^p:= \sum_{\alpha=0}^k \int_0^1 |\nabla^\alpha f(x)|^p\,\di x.
\ea
Equipped with this norm, $W^{k,p}([0,1],\bR)$ is a Banach space. 

We will use the following embedding theorem, see, e.g., Theorem 6.3 in Adams and Fournier \cite{adams2003sobolev}. 
\begin{thm}
\label{t:Sob}
Let $k=1,2$. Then for $p>\frac1k$ it holds that $W^{k,p}([0,1],\bR)$ is continuously embedded in $C([0,1],\bR)$ w.r.t.\ 
the $\sup$-norm, i.e.,
$W^{k,p}([0,1],\bR)\subseteq C([0,1],\bR)$ and there is a constant $C_\text{\rm Sob}(k,p)$ such that for all $f\in W^{k,p}([0,1],\bR)$ 
\ba
\|f\|_\infty \leq C_\text{\rm Sob}(k,p)\|f\|_{ W^{k,p}} .
\ea 
\end{thm}
For $k=1$ and $p=2$ we will use the least (optimal) value of the embedding constant $C_\text{\rm Sob}(1,2)$ found by Marti \cite{marti1983evaluation}:
\ba
\label{e:optSob}
C_\text{\rm Sob}(1,2)=\tanh(1)^{-1/2}\approx 1.1459.
\ea

\medskip
\noindent
\textbf{Acknowlegements.} The text of this paper was partially reviewed for spelling and grammar using ChatGPT 3.5. The plots in Fig.~\ref{f:1} were produced 
with the help of Wolfram Mathematica 13.

\section{Setting and main results\label{s:setting}}

Let $(\Omega,\rF,\bF,\P)$ be a filtered probability space satisfying the usual hypothesis, and let $(B_j)_{j\geq 1}$ and $W$ 
be independent standard Brownian
motions.

For measurable functions 
\ba
F&\colon [0,1]\times[0,1]\times\bR\times\bR\to\bR,\\
f,\sigma &\colon  [0,1]\times\bR\times\bR\to\bR,\\
\Sigma=(\Sigma_1,\Sigma_2,\dots)&\colon [0,1]\times[0,1]\times\bR\times\bR\to\bR^\infty,
\ea
a finite signed measure $\mu$ on $([0,1],\rB([0,1]))$, we consider the following stochastic differential system 
\ba
\label{diffeq} 
X(t,x)&=X_0(x)+ \int^t_0 F(x,\eta(s)|^1_0,X(s,x),Z(s))\,\di s+
 \sum_{j=1}^\infty \int^t_0 \Sigma^j(x,\eta(s)|^1_0,X(s,x),Z(s))\, \di B^j(s), \\
Z(t)&=\int_0^1 X(t,z)\,\mu(\di z),\\
\eta(t)&=\eta_0+\int^t_0f(\eta(s)|^1_0, X(s,\eta(s)|^1_0),Z(s))\,\di s 
 + \int^t_0\sigma(\eta(s)|^1_0, X(s,\eta(s)|^1_0),Z(s))\,
 \di W(s),
\ea
where $X_0\colon[0,1]\to\bR$ is a measurable function, $\eta_0\in \bR$. At first we do not demand, that the ``ice line'' process $\eta$
takes values in the interval $[0,1]$. However since all the functions, and in particular $X(\cdot,x)$ are defined for $x\in [0,1]$,
we truncate the process $\eta$ correspondingly. 

The differential system \eqref{diffeq} can be seen as a continuum of stochastic differential equations parameterized by the coordinate $x\in[0,1]$.
Since it does not contain partial derivatives w.r.t.\ $x$ and is not subject to any boundary conditions at $x=0$ and $x=1$, 
it is not a stochastic partial differential equation. However, the peculiarity 
of this system consists in the composition term $X(t,\cdot)\circ \eta(t)|_0^1$, where the random solution $X$ is evaluated at the random point 
$\eta(t)|_0^1$.

The first result concerns the existence and uniqueness of a (strong) solution of the system \eqref{diffeq}. 
We impose the following set of assumptions \textbf{A}$_1$--\textbf{A}$_3$ on the coefficients of \eqref{diffeq}.
 
\medskip

\noindent 
\textbf{A}$_1$
The functions $f$, $F$ and $\sigma$ 
 as well as the norm
\ba
\|\Sigma(\cdot,\cdot,\cdot,\cdot)\|_{l^2}: = \Big(\sum_{j=1}^\infty \Sigma^j(\cdot,\cdot,\cdot,\cdot)^2 \Big)^{1/2} 
\ea
are Lipschitz continuous with respect to all variables.

\medskip

\noindent 
\textbf{A}$_2$ 
The norms of the derivatives
\ba
&\|\Sigma_x\|_{l^2}, \ \|\Sigma_{X}\|_{l^2},\ \| \Sigma_{xx}\|_{l^2},\
\|\Sigma_{xX}\|_{l^2},\     \|\Sigma_{XX}\|_{l^2}
\ea
are bounded and Lipschitz continuous with respect to all variables, where
\ba
\Sigma_x&=(\Sigma^1_x,\Sigma^2_x,\dots)   ,  &&&  \Sigma_X&=(\Sigma^1_X,\Sigma^2_X,\dots), \\
\Sigma_{xx}&=(\Sigma^1_{xx},\Sigma^2_{xx},\dots), &&&  \Sigma_{xX}&=(\Sigma^1_{xX},\Sigma^2_{xX},\dots),\\
\Sigma_{XX}&=(\Sigma^1_{XX},\Sigma^2_{XX},\dots).
\ea
Moreover, the derivatives
\ba
F_x,\  F_X,\  F_{xx},\ F_{xX},\ F_{XX}  
\ea
are also bounded and Lipschitz continuous with respect to all variables.

\medskip

\noindent 
\textbf{A}$_3$: 
$\mu\colon \rB([0,1])\to\bR$ is a finite signed measure with the total variation $M:=|\mu|([0,1])\in[0,\infty)$.

\begin{thm} 
\label{existence}
Let assumptions $\mathbf{A}_1$, $\mathbf{A}_2$ and $\mathbf{A}_3$ hold true. 
Then for any $X_0(\cdot)\in C^2([0,1],\bR)$ and $\eta_0\in \bR$,
the system \eqref{diffeq} 
has a unique continuous Markovian solution $(X,\eta)$ such that $X(t,\cdot)\in C^2([0,1],\bR)$, $t\geq 0$.
\end{thm}

The next condition guarantees that if $\eta_0\in(0,1)$, then the ice line
does not touch the equator or the North Pole, i.e.\ $\eta(t)\in (0,1)$ for all $t\geq 0$.

\medskip

\noindent 
\textbf{A}$_4$: For all $X,Z\in\mathbb R$ we have 
\ba
\label{e:00}
&f(0,X,Z)\geq 0 \quad\text{and}\quad  f(1,X,Z)\leq 0,\\
&\sigma(0,X,Z)=  \sigma(1,X,Z) =0.
\ea
In particular, the Lipschitz continuity of $f$ and $\sigma$ together with \eqref{e:00} imply that they are bounded, see 
\eqref{e:f01} and \eqref{e:s01} in Section \ref{s:01}.

\begin{thm}
\label{etain01}
Let Assumptions $\mathbf{A}_1$--$\mathbf{A}_4$ hold true.
Then for any $X_0(\cdot)\in C^2([0,1],\bR)$ and $\eta_0\in (0,1)$,
the unique solution $(X,\eta)$ of \eqref{diffeq} from Theorem \ref{existence} satisfies
\ba
\P_{X_0,\eta_0} \Big(\eta(t)=\eta(t)|_0^1\in (0,1),\ t\geq 0\Big)=1.
\ea
\end{thm}
Since the process $\eta(\cdot)$ does not touch the boundaries $x=0$ and $x=1$, we have $\eta(\cdot)|_0^1\equiv \eta(\cdot)$. In the sequel, the truncation 
will be omitted.

The main result of the paper deals with the asymptotic behaviour of the ice line in the limit of the small parameter $\e$ that parameterizes the 
slow ice line dynamics in comparison to the fast fluctuations of the temperature. To this end, we transform the system 
\eqref{diffeq} into a slow-fast system: for $t\geq 0$ and $\e\in(0,1]$ we consider 
\ba
\label{diffeq-eps} 
X^\e(t,x)&=X_0(x)+ \frac{1}{\e} \int^t_0 F(x,\eta^\e(s),X^\e(s,x),Z^\e(s))\,\di s+
\frac{1}{\sqrt{\e}} \sum_{j=1}^\infty \int^t_0 \Sigma^j(x,\eta^\e(s),X^\e(s,x),Z^\e(s))\, \di B^j(s), \\
Z^\e(t)&=\int_0^1 X^\e(t,z)\,\mu(\di z),\\
\eta^\e(t)&=\eta_0+\int^t_0 f(\eta^\e(s), X^\e(s,\eta^\e(s)),Z^\e(s))\,\di s 
 + \int^t_0 \sigma(\eta^\e(s), X^\e(s,\eta^\e(s)),Z^\e(s))\, \di W(s).
\ea

Under Assumption \textbf{A}$_4$, the functions $f$ and $\sigma$ are bounded. We additionally assume 
the boundedness of the diffusion coefficients in the equation for $X^\e$.

\medskip

\noindent 
\textbf{A}$_5$: The norm $\|\Sigma(\cdot,\cdot,\cdot,\cdot)\|_{l^2}$ is bounded.

\medskip

Next we need assumptions, that guarantee the stability of the solution  $X^\e(t,\cdot)$ over the infinite time horizon $t\geq 0$.

\medskip

\noindent 
\textbf{A}$_6$: 
\label{assFXav}
There is $\delta>0$ such that for all $x,\eta\in [0,1]$ and all $X,Z\in \bR$ we have
\ba
\label{e:asm6}
&F_ X(x,\eta, X,Z)   +  \frac32 \|\Sigma_X(x,\eta, X,Z)\|_{l^2}^2  
\leq -\delta <0.
\ea

\medskip

\noindent 
\textbf{A}$_7$: 
There are $A>0$ and $B,C\geq 0$ such that for all $x,\eta\in [0,1]$ and all $X,Z\in \bR$ we have
\ba
X F(x,\eta,X,Z)\leq -AX^2 +BXZ+C
\ea
and
\ba
\label{e:ABC}
A>\frac{B}{2}(1+ C_\mu^2 ), 
\ea
where 
\ba
C_\mu=M\cdot\tanh(1)^{-1/2},
\ea
see \eqref{e:optSob}. Alternatively, if the signed measure $\mu$ has a bounded density, $\mu(\di x)=\mu(x)\,\di x$, $x\in[0,1]$, then we can take
\ba
C_\mu=  \sup_{x\in[0,1]}|\mu(x)|.
\ea
Hence, in the particular case of the Lebesque measure $\mu(\di x)=\di x$, the condition \eqref{e:ABC} turns into $A>B$.

\begin{thm}
\label{wnormmomentsboundthm}
Let Assumptions $\mathbf{A}_1$--$\mathbf{A}_7$ hold true,  $\eta_0\in (0,1)$ and $X_0\in C^2([0,1],\bR)$.
Then there is a constant $C_{X}(1,2)>0$ such that for all $\e\in(0,1]$
\ba
\label{e:estX1}
& \sup_ {t\geq 0} \E_{X_0,\eta_0} \| X^\e(t,\cdot) \|^2_{W^{1,2}} \leq C_{X}(1,2) \Big(1+\| X_0(\cdot)\|^{ 2}_{W^{1,2}} \Big).
\ea
Moreover, there is $p^*>2$ such that for any $p\in[2,p^*)$ there is a constant 
$C_{\nabla X}(1,p^*)>0$ such that for all $\e\in(0,1]$
\ba
\label{e:estnablaX}
& \sup_ {t\geq 0} \E_{X_0,\eta_0} \| \nabla X^\e(t,\cdot)\|^p_{W^{1,p}} 
\leq C_{\nabla X}(1,p^*) \Big(1+\|\nabla X_0(\cdot)\|^{p}_{W^{1,p}} + \| \nabla X_0(\cdot)\|^{2p^*}_{L^{2p^*}} \Big).
\ea
\end{thm}
To prepare the main result, for $\eta\in[0,1]$, $X_0(\cdot)\in C^2([0,1],\bR)$ and $t\geq 0$ we define the frozen processes $\xi^\eta$ and $\zeta^\eta$ 
as solutions of the differential system
\ba
\label{diffeqfrozeneta}
\xi^\eta(t ,x)
&=X_0 (x)+ \int^t_0 F(x,\eta,\xi^\eta (s,x),\zeta^\eta(s))\,\di s
+ \sum_{j=1}^\infty\int_0^t \Sigma^j(x,\eta,\xi^\eta (s,x),\zeta^\eta(s))\, \di B^j(s), \\
\zeta^\eta(t)&=\int_0^1 \xi^\eta (t,z)\,\di z.
\ea
Existence, uniqueness and properties of solutions of \eqref{diffeqfrozeneta} follow from Theorems~\ref{existence} and \ref{etain01} 
with $\e=1$, $f=0$ and $\sigma=0$.

To apply the averaging technique we impose the following assumptions to obtain the weak convergence.
 
\medskip

\noindent
\textbf{A}$_\text{w}$: 
\label{assfhatavwc}
For any $\eta \in [0,1]$ the process $(\xi^\eta(t ,\eta),\zeta^\eta(t))_{t\geq 0}$ is ergodic with the unique 
stationary distribution $\rho^\eta(\di \xi,\di \zeta)$.
Further, for any continuous bounded function $\phi\colon  [0,1]\times\bR\times\bR\to\bR$ and any $X_0(\cdot)\in C^2([0,1],\bR)$ it holds
\ban
\sup_{\eta\in[0,1]}
\E_{X_0,\eta}\Big|\frac{1}{T} \int_0^{T}   \phi (\eta,\xi^{\eta}(s,\eta),\zeta^{\eta} (s))\, \di s
-\iint_{\bR^2}  \phi(\eta,\xi,\zeta)\rho^\eta (\di \xi,\di \zeta) \, \Big|
< \gamma(T)\Big(1+\|X_0(\cdot)\|_{W^{1,2}}\Big),
\ean
where $\gamma(T) \rightarrow  0$ as $T \rightarrow \infty$. 
Let
\ba
\label{e:fsigma}
\widehat f( \eta )&:= \iint_{\bR ^2}  f(\eta, \xi, \zeta)\rho^\eta(\di \xi,\di \zeta), \\
\widehat \sigma^2( \eta )&:= \iint_{\bR ^2}  \sigma^2(\eta, \xi, \zeta)\rho^\eta (\di \xi,\di \zeta),
\ea
and assume that $\widehat f( \cdot)$ and $\widehat\sigma( \cdot )$ are Lipschitz continuous (see Section \ref{s:example} for the concrete example).
 
\begin{thm}
\label{averagingwc}
Let 
assumptions $\mathbf{A}_1$--$\mathbf{A}_7$ and $\mathbf{A}_\text{\emph{w}}$ be satisfied.
Then for any $\eta_0\in (0,1)$ and $X_0(\cdot)\in C^2([0,1],\bR)$ 
\ba
\eta^\e(\cdot) \Rightarrow \widehat \eta(\cdot),\quad \e\to 0,
\ea
in $C([0,1],\bR)$, where $\widehat \eta(\cdot)$ is the unique solution of the SDE
\ba
\label{e:SDEhateta}
\widehat \eta(t)&=\eta_0+\int^t_0 \widehat f( \widehat \eta(s) )\,\di s + \int^t_0 \widehat \sigma(\widehat\eta(s) )\,  \di W(s),
\ea
with the coefficients $\widehat f$ and $\widehat \sigma$ defined in \eqref{e:fsigma}.
\end{thm}

If the diffusion coefficient $\sigma$ in \eqref{diffeq-eps} depends only on $\eta$, $\sigma=\sigma(\eta)$, we can also establish the strong convergence.

\medskip

\noindent
\textbf{A}$_\text{s}$: 
There is a Lipschitz continuous function $\widehat f\colon [0,1]\to \bR$ such that for any $\eta\in [0,1]$, $X_0(\cdot)\in C^2([0,1],\bR)$
\ba
\label{e:fw}
\E_{X_0,\eta}\Big|\frac1T \int_0^T f (\eta,\xi^\eta(s,\eta),\zeta^\eta (s))\,\di s-\widehat f(\eta) \Big| 
< \gamma(T)\Big(1+\|X_0(\cdot)\|_{W^{1,2}}\Big),
\ea
where $\gamma(T) \rightarrow  0$ as $T \rightarrow \infty$. 
 
\begin{thm}
\label{averagingsc}
Let $\sigma=\sigma(\eta)$, and let
assumptions $\mathbf{A}_1$--$\mathbf{A}_7$ and $\mathbf{A}_\text{\emph{s}}$ be satisfied.
Let $\eta_0\in (0,1)$ and $X_0(\cdot)\in C^2([0,1],\bR)$. 
Then for any $T>0$ and $\delta>0$ we have 
\ba
\label{diffetatildeetahat}
\P_{X_0,\eta_0}\Big( \sup_{ t\in [0,T]  }|\eta^\e(t)-\widehat\eta(t)|>\delta\Big)\to 0,\quad \e\to 0,
\ea
where
$\widehat\eta$ is the 
unique solution of the SDE
\ba
\label{e:SDEeta}
\widehat \eta(t)&=\eta_0+\int^t_0 \widehat f( \widehat \eta(s) )\,\di s + \int^t_0 \sigma(\widehat\eta(s) )\,  \di W(s),
\ea
with $\widehat f$ defined in \eqref{e:fw}.
\end{thm}

\section{Example: Stochastic Widiasih's EBM with a moving ice line\label{s:example}}

In this section we fit the deterministic EMB \eqref{e:EBM1}, \eqref{e:EBM-eta} 
into our stochastic setting \eqref{diffeq-eps}. 

First we discuss the equation for $X^\e(\cdot,\cdot)$ and $Z^\e(\cdot)=\int_0^1 X^\e(\cdot,z)\,\mu(\di z)$.

We note that $\mu(\di z)=\di z$, so that $M=1$, and Assumption $\mathbf{A}_3$ is satisfied.

Using the data from Widiasih \cite{widiasih2013dynamics} and Walsh and Rackauckas \cite{walsh2015budyko} we get:
\ba
\label{e:F}
F(x,\eta,X,Z)&=\frac{1}{R}\Big(Q s(x)(1-\alpha(x,\eta))- (a+bX )-c( X -Z)\Big)\\
\ea
where 
\ba
R&=12.6,\quad Q=343,\\
s(x)&=1+\frac{s_2}{2}(3x^2-1),\quad s_2=-0.482,\\
\alpha(x,\eta)&= \frac{\alpha_s+\alpha_w}{2}  + \frac{\alpha_s-\alpha_w}{2} \tanh(K(x-\eta)),\\
\alpha_w&=0.32,\quad \alpha_s = 0.62,\quad K=25,\\
 a &= 202, \quad b = 1.9, \quad c = 3.04.
\ea
In order to use the notation of Assumption $\mathbf{A}_7$ rewrite \eqref{e:F} as
\ba
F(x,\eta,X,Z)=-A X +B Z + h(x,\eta)
\ea
with a bounded function 
\ba
h(x,\eta)&=\frac1R\Big(Q s(x)(1-\alpha(x,\eta))- a\Big).
\ea
and the constants
\ba
A&=\frac{b+c}{R},\quad  B=\frac{c}{R}.
\ea
Hence, $F$ satisfies Assumptions $\mathbf{A}_1$ and $\mathbf{A}_2$.

Further we assume that $X$ is driven by only one Brownian motion $B(\cdot)$, and 
the diffusion coefficient $\Sigma= \Sigma(x,\eta)$ does not depend on $X$ and $Z$, and satisfies Assumptions $\mathbf{A}_1$, 
$\mathbf{A}_2$ and $\mathbf{A}_5$. In our numeric example we set
\ba
\Sigma(x,\eta)&:=\frac{2+\eta}{1+x^2}.
\ea
Since $F_X=-A$ and $\Sigma_X=0$, Assumption $\mathbf{A}_6$ is satisfied.

For any $\kappa\in (0,b/R)$ we estimate 
\ba
X F(x,\eta,X,Z)&=-A X^2 +B X Z + X h(x,\eta)= -(A-\kappa) X^2 + BXZ -\kappa X^2  + X h(x,\eta)\\
&\leq -(A-\kappa) X^2 + BXZ +\max_{X}\Big(-\kappa X^2  + |X|\cdot  \|h(\cdot,\cdot)\|_\infty\Big),
\ea
so that  $\mathbf{A}_7$ is satisfied, too, with $A-\kappa$ taken instead of $A$.

Now we inspect the equation for $\eta^\e$ in \eqref{diffeq-eps}.
We note that in the equation \eqref{e:EBM-eta} in the deterministic model, the drift term $f(\eta,X,Z)=X-X_\text{critical}$
does not satisfy Assumption $\mathbf{A}_4$.
Hence we modify the function $f$ in such a way that $f(\eta,X,Z)\approx X-X_\text{critical}$ away from the boundaries $x=0$ and $x=1$, 
and $f(0,X,Z)> 0$, $f(1,X,Z)< 0$.
More precisely, we set 
\ba
f(\eta,X,Z)=f(\eta,X)&:=-\kappa\Big(\eta -\frac12\Big) +  K \Big(1-\ex^{-K\eta (1-\eta )}\Big) 
\arctan\Big(\frac{X-X_\text{critical}}{K}\Big),\\
\quad K&=25,\ \kappa=0.1,\  X_\text{critical}=-10^\circ \text{C}.
\ea
In our numerical example we set
\ba
\sigma(\eta,X,Z)=\sigma(\eta)&:= \eta(1-\eta).
\ea
The functions $f$ and $\sigma$ satisfy Assumptions $\mathbf{A}_1$ and $\mathbf{A}_4$.

The frozen processes \eqref{diffeqfrozeneta} take the form
\begin{align}
\label{e:xi}
\xi^\eta(t,x)&= X_0(x) +  \int^t_{0}\Big( -A \xi^\eta (s,x)+ B \zeta^\eta(s) + h(x,\eta )\Big)\,\di s
 + \Sigma(x,\eta)  B(t),\\
\label{e:barxi}
\zeta^\eta(t)&= \int_0^1 \xi^\eta(t,z)\,\di z
 = \bar{X}_0+  \int^t_{0}\Big( -(A-B) \zeta^\eta(s) + \bar h (\eta)\Big)\,\di s+ \bar{\Sigma}(\eta)  B(t),
\end{align}
where the following notation is used:
\ba
\bar h (\eta) =\int_0^1 h (\eta,z)\,\di z ,\quad \bar{\Sigma}(\eta)   =\int_0^1 \Sigma (\eta,z)\,\di z, \quad \bar{X}_0   =\int_0^1 X_0 ( z)\,\di z.
\ea
The system \eqref{e:xi}, \eqref{e:barxi} can be solved in the closed form.
The process $\zeta^\eta$ is an Ornstein--Uhlenbeck process whose solution can be written explicitly:
\ba
\label{e:zeta}
\zeta^\eta(t)
&=\bar{X}_0\ex^{-(A-B)t} + \bar{h}(\eta)\frac{1-\ex^{-(A-B)t}}{A-B}+ \bar\Sigma(\eta)\int_0^t \ex^{-(A-B)(t-s)}\,\di B(s) . 
\ea
After substituting $\zeta^\eta$ into \eqref{e:xi}, we solve linear equation \eqref{e:xi} w.r.t.\ $\xi^\eta$ to obtain  
\ba
\xi^\eta(t,x)&= X_0\ex^{-At} 
+ \int_0^t \ex^{-A(t-s)} \Big( -B \zeta^\eta(s) +h(\eta,x)\Big)   \,\di s
+ \Sigma(\eta,x) \int_0^t \ex^{-A(t-s)}  \,\di B(s).
\ea
Applying the stochastic Fubini theorem we establish the equality
\ba
\int_0^t \ex^{-A(t-s)} \Big[ \int_0^s \ex^{-(A-B)(s-r)}\,\di B(r) \Big]    \,\di s
= \frac{1}{B}\int_0^t \Big(\ex^{-(A-B)(t-s)} -  \ex^{-A(t-s)}  \Big)\,\di B(s).
\ea
After calculating the integrals we finally obtain the solution:
\ba
\xi^\eta(t,x)=X_0(x)\ex^{-At} &+ h(\eta,x) \frac{1-\ex^{-At}}{A}-\Big(\bar{X}_0-\frac{\bar{h}(\eta)}{A-B}\Big)\ex^{-At}(1-\ex^{Bt})
+\frac{B}{A-B}\bar{h}(\eta)\frac{1-\ex^{-At}}{A}\\
& + \Sigma(\eta,x)\int_0^t \ex^{-A(t-s)}\,\di B(s) - \bar\Sigma(\eta)\int_0^t \Big( \ex^{-A(t-s)}-  \ex^{-(A-B)(t-s)}    \Big)\,\di B(s).
\ea
It is clear that for each $\eta\in[0,1]$ the process $(\xi^\eta(t,\eta),\zeta^\eta(t))_{t\geq 0}$ is Gaussian, and
\ba
(\xi^\eta(t,\eta),\zeta^\eta(t))\stackrel{\di}{\to} (\xi^\eta_\infty,\zeta^\eta_\infty), 
\ea
where $(\xi^\eta_\infty,\zeta^\eta_\infty)$ has the bivariate Gaussian distribution $\rho^\eta(\di\xi,\di \zeta)$ with the
following characteristics:
\ba
\E \xi^\eta_\infty &= \frac{h(\eta,\eta)}{A}+ \frac{B}{A-B}\frac{\bar{h}(\eta) }{A},\\
\E \zeta^\eta_\infty & = \frac{\bar{h}(\eta) }{A-B},\\
\Var \xi^\eta_\infty&=\frac{(\Sigma(\eta,\eta)-\bar\Sigma(\eta))^2}{2A} 
+  \frac{2(\Sigma(\eta,\eta)-\bar\Sigma(\eta))\bar\Sigma(\eta)}{2A-B}  +\frac{\bar\Sigma(\eta)^2}{2(A-B)},\\
\Var\zeta^\eta_\infty&=\frac{\bar\Sigma(\eta)^2}{2(A-B)}, \\
\Cov( \xi^\eta_\infty,\zeta_\infty^\eta)
&= \frac{(\Sigma(\eta,\eta)-\bar\Sigma(\eta))\bar\Sigma(\eta)}{2A-B}   +\frac{\bar\Sigma(\eta)^2}{2(A-B)}.
\ea

Since the diffusion coefficient $\sigma=\sigma(\eta)$ does not depend on $Z$, we will apply Theorem \ref{averagingsc} and establish strong convergence 
of $\eta^\e$. 
Since the drift $f=f(\eta,X)$ does not depend on $Z$, in Assumption \textbf{A}$_\text{s}$
and in the formulas \eqref{e:fsigma} it is sufficient to take expectations with respect to the probability 
law of the process $\xi^\eta(\cdot,\eta)$ and of the limit 
$\xi^\eta_\infty$, respectively. Note that under our assumptions the probability laws of these random variables are not degenerate.

A straightforward integration of the drift $f$
with respect to the Gaussian density of $\xi^\eta(t,\eta)$
shows that Assumption \textbf{A}$_\text{s}$ holds true:
there are constants $c_1,c_2>0$ such that for all $\eta\in[0,1]$ and $t\geq 0$
\ba
\label{konfrozeneta}
\E_{X_0,\eta}\Big|\frac{1}{T} \int_0^{T} f(\eta,\xi^\eta(s,\eta))\,\di s
-\widehat f(\eta) \Big|
&\leq \frac{c_1}{T+1}\Big(1+\|X_0(\cdot)\|_\infty\Big)\\
&\leq \frac{c_2}{T+1}\Big(1+\|X_0(\cdot)\|_{W^{1,2}}\Big),
\ea
where
\ba
\widehat f (\eta)= \E f(\eta,\xi^\eta_\infty)= 
\int_\bR   \frac{ f(\eta,\xi)}{\sqrt{2\pi \Var\xi^\eta_\infty}}\exp\Big(-\frac{(\xi- \E \xi^\eta_\infty )^2}{ 2\Var \xi^\eta_\infty}\Big)\,\di \xi.
\ea
It is easy to see that $\eta\mapsto\widehat f(\eta)$ is Lipschitz continuous. 
Hence, Theorem \ref{averagingsc} yields that 
\ba
\eta^\e\stackrel{\text{u.c.p.}}{\to} \widehat\eta,\quad \e\to 0,
\ea
where $\widehat \eta$ is a solution of \eqref{e:SDEeta}.

On Fig.~\ref{f:1} we depict the function $\widehat f$ for various values of the mean solar incoming radiation $Q$.
\begin{figure}
\begin{center}
\includegraphics[width=0.25\textwidth]{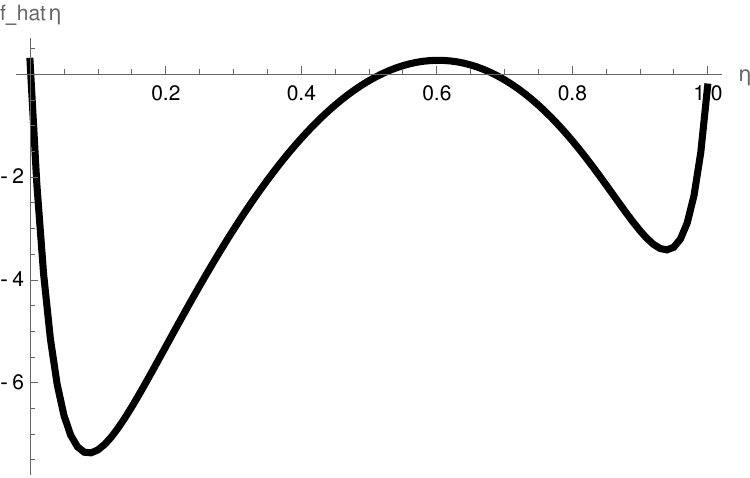}\hspace{1cm}
\includegraphics[width=0.25\textwidth]{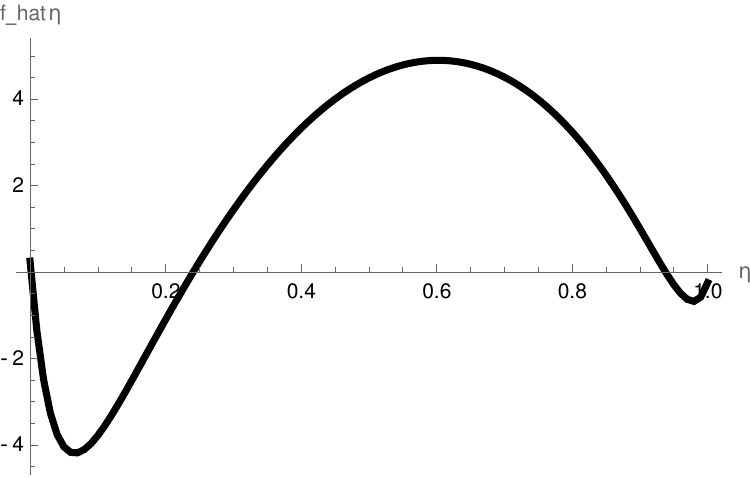}\hspace{1cm}
\includegraphics[width=0.25\textwidth]{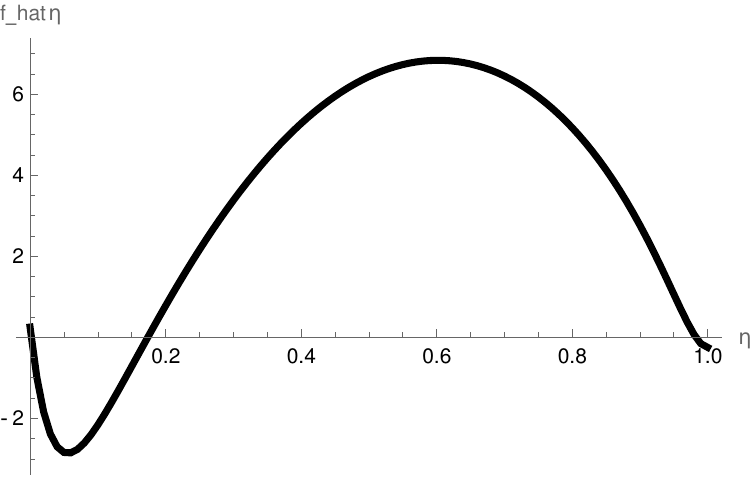}
\end{center}
\caption{The drift coefficient $\widehat f$ for 
the incoming solar radiation $Q=327\, W/m^2$ (left), $Q=343\,  W/m^2$, i.e.\ for present climate conditions (middle), and $Q=350\,  W/m^2$ (right). 
\label{f:1}} 
\end{figure}
It is instructive to notice that under our assumptions, the limit dynamics of $\widehat \eta$ is confined to the interval $(0,1)$. The 
SDE \eqref{e:SDEeta} has the stationary distribution 
\ba
\widehat p_\text{st}(\eta)=\frac{N}{\sigma(\eta)^2}\exp\Big( 2\int_{1/2}^\eta \frac{\widehat f(u)}{\sigma(u)^2}\,\di u\Big),\quad \eta\in (0,1),
\ea
$N>0$ being the normalizing constant.
The explicit formula for the stationary density allows to determine the most probable metastable 
locations of the ice line $\widehat\eta$ as well as its mean value.
One can also calculate the mean transition times between the metastable climate states.

\section{Proof of Theorem \ref{existence}\label{s:existence}}

\subsection{Solution of the truncated system}

Let 
$x\mapsto X(x)$ be a $C^1$-function and let $\mu$ be a finite signed measure. Then with the help of Fubini's theorem we 
obtain the following formula:
\ba
\label{e:XXx}
\int^1_0  X(z)\, \mu (\di z)&=\int_{[0,1]} \Big(  X(0)  +\int_{[0,1]}\bI_{[0,z]}(y) \nabla  X(y)\, \di y\Big)\, \mu (\di z)\\
&= X(0)\mu ([0,1]) + \int_{[0,1]}  \nabla  X(y)\Big[ \int_{[0,1]}  \bI_{[y,1]}(z) \, \mu (\di z)\Big]\, \di y\\
&= X(0)\mu ([0,1]) + \int_{[0,1]}  \nabla  X(y)\mu([y,1])\, \di y.
\ea
In case of absolutely continuous $\mu$ with a bounded density $\mu(z)$ we have
\ba
\label{e:Z1}
\Big|\int_0^1 X(z)\mu(z)\,\di z\Big|^2\leq \sup_{z\in[0,1]}|\mu(z)|^2 \cdot \int_0^1 |X(z)|^2\, \di z 
\leq \sup_{z\in[0,1]}|\mu(z)|^2 \cdot
\|X(\cdot)\|_{W^{1,2} }^2.
\ea
In the case of arbitrary finite signed measure $\mu$ we estimate with the help of the Sobolev embedding theorem~\ref{t:Sob}: 
\ba
\label{e:Z2}
\Big| \int_0^1 X(z)\mu(\di z)\Big|^2 
\leq M^2\cdot \|X(\cdot)\|_\infty^2 
\leq M^2 \cdot C_\text{Sob}(1,2)^2 \cdot \|X(\cdot)\|_{W^{1,2} }^2.
\ea
We unite the estimate \eqref{e:Z1} and \eqref{e:Z2} in one estimate with the constant $C_\mu$:
\ba
\label{e:ZZZ}
\Big|\int_0^1 X(z)\,\mu(\di z)\Big|^2\leq C_\mu^2\cdot \|X(\cdot)\|_{W^{1,2} }^2.
\ea
As it was explained in the Introduction, one of the difficulties of the analysis of system \eqref{diffeq}
consists in considering the composition $X(\cdot,\eta(\cdot))$. It is clear, that the random 
field $x\mapsto X(\cdot,x)$ has to be   
regular enough, e.g., Lipschitz continuous with a Lipschitz constant bounded in a 
certain sense in order to guarantee the existence and uniqueness of the solution $(X,\eta)$. 

To overcome this difficulty, first we modify the system \eqref{diffeq} into a system in which the gradient $\nabla X$ is bounded.
Recall that the gradient $\nabla X$ can be obtained by formal differentiation of the SDE for $X$. It  
satisfies the variational equation
\ba
\nabla X(t,x)&=\nabla  X_0(x) +\int_0^t  \Big( F_x +  F_ X  \nabla X(s,x)\Big) \,\di s 
+ \sum_{j=1}^\infty \int_0^t  \Big( \Sigma^j_x+  \Sigma^j_X \nabla X(s,x)\Big)\, \di B^j(s).
\ea
We set up an auxiliary system of equations with an unknown function $Y^\phi$ standing for the
gradient $\nabla X$ truncated by a bounded cut-off function $\phi\colon \bR\to\bR$ such that $\phi(x)=x$ on a large interval
around zero.

Then the following modifications will be made. First, we will evaluate the composition $X(t,\eta)$, $\eta\in[0,1]$, with the help of the formula
\ba
\label{e:XX}
X(t,\eta)=X(t,0)+\int_0^\eta \nabla X(t,y)\,\di z.
\ea
Second, we replace $\nabla X$ by the process $Y^\phi$ which is a solution of the ``truncated'' variational equation. Finally, 
we use the bounded process $\phi(Y^\phi)$ instead of $\nabla X$ in \eqref{e:XX}. The averaged process $Z$ is treated similarly with the help 
of the formula \eqref{e:XXx}.

More presicely, let $\phi$ satisfy the following conditions:
\ba
\label{e:phi}
&\phi\in C^2(\bR,\bR),\\
&\phi(0)=0,\\
&|\phi(x)|\leq C_\phi\quad \text{ for some }C_\phi>0,\\
&|\phi(x)-\phi(y)|\leq   |x-y|,\quad x,y\in\bR,\\
&|\phi'(x)-\phi'(y)|\leq   |x-y|,\quad x,y\in\bR.
\ea
In particular, we have that $\|\phi'\|_\infty\leq 1$, $\|\phi''\|_\infty\leq 1$ and $|\phi(x)|\leq |x|$.
A concrete example of such function $\phi$ is given in \eqref{e:phi-exa}.  

Consider the following truncated and extended counterpart of the system 
\eqref{diffeq}:
\ba
\label{diffeqtrun}
& X^\phi (t,x)=X_0(x)+\int_0^t  F \Big(x,  \eta^\phi (s)|^1_0, X^\phi (s,x), Z^\phi (s)  \Big)\,\di s +
\sum_{j=1}^\infty \int_0^t  \Sigma^j \Big(x,  \eta^\phi (s)|^1_0,  X^\phi (s,x), Z^\phi (s) \Big)\, \di B^j(s)\\
& Y^\phi(t,x)=\nabla X_0(x) +\int_0^t \Big( F_x \big(x,  \eta^\phi (s)|^1_0, X ^\phi(s,x), Z^\phi (s) \big)
+  F_ X \big(x,  \eta^\phi (s)|^1_0, X ^\phi(s,x), Z^\phi (s) \big)\phi ( Y^\phi (s,x) )\Big)\,\di s\\
&\qquad +\sum_{j=1}^\infty 
\int_0^t \Big(\Sigma^j_x(x,  \eta^\phi (s)|^1_0,  X^\phi(s,x), Z^\phi (s))
+\Sigma^j_X \Big(x,  \eta^\phi (s)|^1_0, X^\phi (s,x), Z^\phi (s) \Big)\phi ( Y ^\phi(s,x)) \Big)\, \di B^j(s),\\
& Z^\phi(s)= X^\phi  (s,0)  \mu ([0,1]) + \int^1_0  \phi ( Y ^\phi(s,z))  \mu ([z,1])\, \di z\\
&  \eta^\phi (t)= \eta _0
+\int_0^t f\Big(  \eta^\phi (s)|^1_0,  X^\phi (s,0)  +\int _0^{  \eta^\phi (s)|^1_0} \phi ( Y^\phi(s,z))\, \di z, Z^\phi (s) \Big)\,\di s\\
&\qquad \qquad+ \int_0^t    
\sigma\Big(  \eta^\phi (s)|^1_0,   X^\phi (s,0) +\int _0^{  \eta^\phi (s)|^1_0} \phi ( Y^\phi(s,z))\, \d z, Z^\phi (s) \Big)\,\di W(s),\\
\ea

\begin{thm}
 \label{existence-phi}
Let assumptions $\mathbf{A}_1$, $\mathbf{A}_2$ and $\mathbf{A}_3$ hold true and let $\phi$ satisfy conditions \eqref{e:phi}. 
Then for any $X_0(\cdot)\in C^2([0,1],\bR)$ and $\eta_0\in \bR$,
the system \eqref{diffeqtrun} 
has a unique continuous solution $(X^\phi, Y^\phi,\eta^\phi )$ such that $X^\phi (t,\cdot)\in C^2([0,1],\bR)$, $t\geq 0$.
\end{thm}

The existence and uniqueness of a solution of \eqref{diffeqtrun} will be established by the method of successive approximations
with the help of the estimates obtained in the following Lemmas.

We set $X_0^\phi:=X_0$, $Y_0^\phi:=\nabla X_0$, $\eta_0^\phi:=\eta_0$, and for $k\geq 0$ we define
\ba
X_{k+1}^\phi(t,x)&=X_0(x)+\int_0^t  F \Big(x,  \eta^\phi_k(s)|^1_0, X_k^\phi(s,x), Z_k^\phi(s)  \Big)\,\di s  
+\sum_{j=1}^\infty\int_0^t  \Sigma^j \Big(x,  \eta^\phi_k(s)|^1_0,  X_k^\phi(s,x), Z_k^\phi(s)  \Big)\, \di B^j(s),\\
Y_{k+1}^\phi(t,x)&
=\nabla X_0(x) +\int_0^t \Big( F_x(x,  \eta^\phi_k(s)|^1_0, X_k^\phi(s,x), Z_k^\phi(s))
+  F_X(x,  \eta^\phi_k(s)|^1_0, X_k^\phi(s,x), Z_k^\phi(s))\phi ( Y_{k}^\phi(s,x) )\Big)\,\di s\\
&\qquad 
+\sum_{j=1}^\infty\int_0^t \Big(\Sigma^j_x(x,  \eta^\phi_k(s)|^1_0,  X_k^\phi(s,x), Z_k^\phi(s))
+ \Sigma^j_X(x,  \eta^\phi_k(s)|^1_0, X_k^\phi(s,x), Z_k^\phi(s))\phi ( Y_{k}^\phi(s,x)) \Big)\, \di B^j(s),\\
Z_k^\phi(s)&= X_k (s,0)  \mu ([0,1]) + \int^1_0  \phi ( Y_k ^\phi(s,z))  \mu ([z,1])\, \di z,\\
\eta^\phi_{k+1}(t) &= \eta_0+\int_0^t  
f\Big(  \eta^\phi_k(s)|^1_0,  X^\phi_k(s,0)   +\int _0^{  \eta^\phi_k(s)|^1_0} \phi ( Y_k^\phi(s,z))\,\di z, Z_k^\phi(s) \Big)\,\di s\\
&\qquad\qquad + \int_0^t    \sigma\Big(  \eta^\phi_k(s)|^1_0,  X^\phi_k(s,0) 
+\int_0^{\eta^\phi_k(s)|^1_0} \phi ( Y_k^\phi(s,z))\,\di z, Z_k^\phi(s) \Big)\,  \di W(s).
\ea
The following elementary estimates hold true due to the bound $|\phi(y)|\leq |y|$, $y\in \bR$:
\ba
\label{e:est}
&\Big|X^\phi_k(s,0)   +\int _0^{  \eta^\phi_k(s)|^1_0} \phi ( Y_k^\phi(s,z))\,\di z\Big|\leq |X^\phi_k(s,0)| + \int_0^1 |Y_k^\phi(s,z)|\,\di z ,\\
&|Z_k^\phi(s)|\leq  M |X_k (s,0)| + M \int_0^1 |Y_k^\phi(s,z)|\,\di z .
\ea

\begin{lem} 
\label{l:boun}
Let $T>0$, $p\geq 2$, $\eta_0\in\bR$ and $X_0(\cdot)\in C^1([0,1],\bR)$.
For each $k\geq 0$, the processes $(t,x)\mapsto X_{k}^\phi(t,x)$, $(t,x)\mapsto Y_{k}^\phi(t,x)$ 
are continuous for $x\in[0,1]$ and $t\in[0,T]$ and the process $t\mapsto \eta^\phi_k(t)$ is continuous
for $t\in[0,T]$, 
and there is $C(p, T)>0$ that does not depend on $C_\phi$ such that
\ba
\label{e:pk}
\E_{X_0,\eta_0 }\Big[
&\sup_{0\leq t\leq T} |\eta^\phi_{k}(t)|^p
+\sup_{0\leq t\leq T}|X^\phi_{k}(t,0)|^p
+\sup_{0\leq t\leq T}\int_0^1 |X^\phi_{k}(t,x)|^p\,\di x
+\sup_{0\leq t\leq T}\int_0^1 |Y^\phi_{k}(t,x)|^p\,\di x
\Big]\\
&\leq C(p,T) \Big(|\eta_0|^p
+|X_0(0)|^p
+\int_0^1 |X_0(x)|^p\,\di x
+\int_0^1 |\nabla X_0(x)|^p\,\di x
\Big).
\ea
\end{lem}
\begin{proof}
Let $T>0$ and $p\geq 2$ be fixed. Obviously, for $k=0$ the processes
\ba
\eta^\phi_{1}(t) &= \eta_0+\int_0^t  
f\Big(  \eta^\phi_0|^1_0,  X_0(0)   +\int _0^{\eta_0|^1_0} \phi (\nabla X_0(z))\,\di z, Z_0^\phi \Big)\,\di s\\
&\qquad\qquad + \int_0^t    \sigma\Big(  \eta^\phi_k|^1_0,  X_0(0) 
+\int_0^{\eta^\phi|^1_0} \phi (\nabla X_0(z))\,\di z, Z_0^\phi \Big)\,  \di W(s),\\
X_{1}^\phi(t,x)&=X_0(x)+\int_0^t  F \Big(x,  \eta_0|^1_0, X_0(x), Z_0^\phi \Big)\,\di s  
+\sum_{j=1}^\infty\int_0^t  \Sigma^j \Big(x,  \eta_0|^1_0,  X_0(x), Z_0^\phi  \Big)\, \di B^j(s),\\
Y_{1}^\phi(t,x)&
=\nabla X_0(x) +\int_0^t \Big( F_x(x,  \eta_0|^1_0, X_0(x), Z_0^\phi)
+  F_X(x,  \eta_0|^1_0, X_0(x), Z_0^\phi)\phi ( \nabla X_0(x) )\Big)\,\di s\\
&\qquad 
+\sum_{j=1}^\infty\int_0^t \Big(\Sigma^j_x(x,   \eta_0|^1_0, X_0(x), Z_0^\phi   )
+ \Sigma^j_X(x, \eta_0|^1_0, X_0(x), Z_0^\phi)\phi (\nabla X_0(x)) \Big)\, \di B^j(s),\\
Z_0^\phi&= X_0 (0)  \mu ([0,1]) + \int^1_0  \phi ( \nabla X_0(z))  \mu ([z,1])\, \di z
\ea
are well defined, continuous by the Kolmogorov criterion 
(see, e.g., Theorem 1.4.1 in Kunita \cite{Kunita97}) and \eqref{e:pk} holds true.

Assume by 
induction that \eqref{e:pk} holds for some $k\geq 1$.
By the Lipschitz continuity, for all $x\in[0,1]$ and $\eta,X,Z\in\bR$ we have 
\ba
|f(\eta|_0^1,X,Z)|+ |\sigma(\eta|_0^1,X,Z)|+|F(x,\eta|_0^1,X,Z)|+\|\Sigma(x,\eta|_0^1,X,Z)\|_{l^2}\leq C(1+|X|+|Z|).
\ea
for some $C>0$.

With the help of \eqref{e:est} and the Burkholder inequality we get
\ba
&\E_{X_0,\eta_0} \sup_{s\leq t}|\eta^\phi_{k+1}(s)|^p \\
&\leq 
3^{p-1} |\eta_0|^p 
+ 3^{p-1}\E_{X_0,\eta_0} \sup_{s\leq t} \Big|\int_0^s f\Big(  \eta^\phi_k(u)|^1_0,  X^\phi_k(u,0) 
+\int_0^{  \eta^\phi_k(u)|^1_0} \phi ( Y_k^\phi(u,z))\,\di z, Z_k^\phi(u) \Big)\,\di u\Big|^p\\
&+3^{p-1} \E_{X_0,\eta_0}\sup_{s\leq t} \Big|\int_0^s    \sigma\Big(  \eta^\phi_k(u)|^1_0,  X^\phi_k(u,0) 
+\int_0^{\eta^\phi_k(u)|^1_0} \phi ( Y_k^\phi(u,z))\,\di z, Z_k^\phi(u) \Big)\,  \di W(u)|^p\\
&\leq 3^{p-1}|\eta_0|^p  +
3^{p-1}T^{p-1}
\int_0^t \E_{X_0,\eta_0} \Big|f\Big(  \eta^\phi_k(u)|^1_0,  X^\phi_k(u,0)
+\int _0^{  \eta^\phi_k(u)|^1_0} \phi ( Y_k^\phi(u,z))\,\di z, Z_k^\phi(u) \Big)\Big|^p\,\di u\\
&+3^{p-1}T^{\frac{p}{2}-1}C_\text{Burk} \int_0^t    \E_{X_0,\eta_0}  \Big|\sigma\Big(  \eta^\phi_k(u)|^1_0,  X^\phi_k(u,0) 
+\int_0^{\eta^\phi_k(u)|^1_0} \phi ( Y_k^\phi(u,z))\,\di z, Z_k^\phi(u) \Big)\Big|^p\,  \di u\\
&\leq 3^{p-1}|\eta_0|^p + C_1 + C_1 \int_0^t \E_{X_0,\eta_0} \sup_{u\leq s}|X^\phi_k(u,0)|^p\,\di s
+ C_1 \int_0^t \E_{X_0,\eta_0} \sup_{u\leq s} \int_0^1 |Y^\phi_k(u,x)|^p\,\di x\, \di s
<\infty.
\ea
Note that the constant $C_1$ does not depend on $k$ and $C_\phi$. 

Moreover, we can also estimate
\ba
\label{e:est1}
\E_{X_0,\eta_0} \sup_{s\leq t}|\eta^\phi_{k+1}(s)|^p 
\leq 3^{p-1}|\eta_0|^p+ C_1
&+ C_1 \int_0^t \max_{0\leq j\leq k+1} \E_{X_0,\eta_0} \sup_{0\leq u\leq s}|X^\phi_j(u,0)|^p\,\di s\\
&+ C_1 \int_0^t \max_{0\leq j\leq k+1}\E_{X_0,\eta_0} \sup_{0\leq u\leq s} \int_0^1 |Y^\phi_j(u,x)|^p\,\di x\, \di s.
\ea
Analogously we obtain that
\ba
\E_{X_0,\eta_0} \sup_{0\leq s\leq t}\int^1_0 \Big|  X^\phi_{k+1}(s,x)\Big|^p \, \di x
&
\leq 3^{p-1} \int^1_0 |  X^\phi_{0}(x) |^p \, \di x +  C_2\\
&+C_2 \int_0^t  \E_{X_0,\eta_0} \sup_{0\leq u\leq s} \int^1_0  |X_{k}^\phi(u,x)|^p \,\di x \, \di u\\
&+C_2 \int_0^t \E_{X_0,\eta_0} \sup_{u\leq s} \int_0^1 |Y^\phi_k(u,x)|^p\,\di x\, \di s<\infty.
\ea
Finally using the boundedness of $F_x$, $F_{X}$, $\|\Sigma_x\|_{l^2}$ and $\|\Sigma_X\|_{l^2}$, and the equality $|\phi(x)|\leq |x|$ we get
\ba
\label{e:Y}
\E_{X_0,\eta_0}\sup_{0\leq s\leq t} \int_0^1  |Y^\phi_{k+1}(s,x)|^p\,\di x
&=3^{p-1}\int_0^1 |\nabla X^\phi_{0}(x)|^p\,\di x 
+ C_3\\
&+ C_3 \int_0^t \E_{X_0,\eta_0} \sup_{0\leq s\leq t}\int^1_0 |  Y^\phi_{k}(u,x)|^p\, \di u \, \di x<\infty,
\ea
where the constants $C_2$ and $C_3$ do not depend on $k$ and $C_\phi$.

For $t\in[0,T]$, denote
\ba
H_k(t):=
\max_{0\leq j\leq k} \E_{X_0,\eta_0}\Big[\sup_{0\leq s\leq t}\int_0^1 |X^\phi_j(s,x)|^p\,\di x 
&+\sup_{0\leq s\leq t}|X^\phi_j(s,0)|^p\\
&+\sup_{0\leq s\leq t}\int_0^1 |Y^\phi_j(s,x)|^p\,\di x
+\sup_{0\leq s\leq t} |\eta^\phi_j(s)|^p\Big].
\ea
Hence writing estimates analogous to \eqref{e:est1}, for each $k\geq 0$ we get the inequality
\ba
H_k(t)\leq 3^{p-1} H_0 +
C_4+C_4\int_0^t H_k(s)\,\di s,
\ea
where $C_4=C_1+C_2+C_3$ do not depend on $k$ and $C_\phi$. Hence for some $C(p,T)>0$ that does not depend on $C_\phi$ we get
\ba
\sup_{k\geq 0} H_k(T)\leq C(p,T) H_0.
\ea
The continuity of all the processes for $k\geq 2$ follows by the Kolmogorov criterion.
\end{proof}

\begin{lem}
\label{l:int}
Let $f=f(t,x,\omega)$ be a non-negative continuous function of $(t,x)$.
Then
\ba
\E\sup_{s\leq t}\int_0^1 \int_0^s f(u,x)\,\di u\,\di x
&= \E\int_0^1 \sup_{s\leq t}\int_0^s f(u,x)\,\di u\,\di x\\
&=\int_0^1 \E\sup_{s\leq t}\int_0^s f(u,x)\,\di x\,\di u
=\int_0^t \E \int_0^1 f(u,x)\,\di x\,\di u.
\ea
\end{lem}
\begin{proof}
For each $\omega\in\Omega$, with the help of Fubini's theorem we have that
\ba
\sup_{s\leq t}\int_0^1 \int_0^s f(u,x,\omega)\,\di u\,\di x
&=\sup_{s\leq t}\int_0^s \int_0^1 f(u,x,\omega)\,\di x\,\di u\\
&=\int_0^t \int_0^1 f(u,x,\omega)\,\di x\,\di u\\
&=\int_0^1 \int_0^t f(u,x,\omega)\,\di u\,\di x
=\int_0^1 \sup_{s\leq t}\int_0^s f(u,x,\omega)\,\di u\,\di x.
\ea
Taking expectation and applying again Fubini's theorem to the second equality yields the statement.
\end{proof}

\begin{lem} 
\label{W1ptrun}
Let Assumptions $\mathbf{A}_1$, $\mathbf{A}_2$ and $\mathbf{A}_3$ hold true and let the function $\phi$ satisfy \eqref{e:phi}. Let
$X_0(\cdot)\in C^2([0,1],\bR)$, and $\eta_0\in\bR$.
Then the system
\eqref{diffeqtrun} has a unique strong solution with a modification $X^\phi(t,\cdot)\in C^2([0,1],\bR)$ and $Y^\phi(t,\cdot)\in C^1([0,1],\bR)$ 
for $t\geq 0$. 
\end{lem}

\begin{proof}
Let $T>0$ and $p\geq 2$. For $0\leq t\leq T$ and $k\geq 0$ denote
\ba
\label{bigsystem}
G_{k+1}(t)&:=
\E_{X_0,\eta_0}\sup_{0\leq s\leq t} |\eta^\phi_{k+1}(s) -  \eta^\phi_{k}(s) |^p
+\E_{X_0,\eta_0}\sup_{0\leq s\leq t}|X^\phi_{k+1}(s,0) -  X^\phi_k (s,0) |^p\\
&+\E_{X_0,\eta_0}\sup_{0\leq s\leq t}\int_0^1 |X^\phi_{k+1}(s,x) -  X^\phi_k (s,x) |^p\,\di x
+\E_{X_0,\eta_0}\sup_{0\leq s\leq t}\int_0^1 | Y^\phi_{k+1}(s,x) -  Y^\phi_k (s,x) |^p\,\di x. 
\ea
We estimate the summands in \eqref{bigsystem}
one by one with the help of the Burkholder inequality for stochastic integrals 
and the Lipschitz property of the functions standing in the integrands, and the 
elementary inequality 
\ba
\Big|x|_0^1-y|_0^1\Big|\leq |x-y|,\quad x,y\in\bR.
\ea
The constants $C_1,C_2,\dots$ in the sequel will not depend on $k$ but may depend on $C_\phi$.

\noindent
1. We have
\ba
&\E_{X_0,\eta_0}\sup_{0\leq s\leq t}  \int^1_0 |  X^\phi_{k+1}(s,x) -  X^\phi_k (s,x)|^p\, \di x\\
&\leq 2^{p-1}\E_{X_0,\eta_0}\sup_{0\leq s\leq t}\int^1_0 \Big| \int_0^s\Big( F (x,  \eta^\phi_k(u)|_0^1, X_k^\phi(u,x), Z_k^\phi(u) )\\
&\qquad\qquad \qquad\qquad \qquad\qquad \qquad\qquad 
- F  (x,  \eta^\phi_{k-1}(u)|_0^1 , X_{k-1}^\phi(u,x), Z_{k-1}^\phi(u)   )\Big)\, \di u \Big| \,\di x\\
&+2^{p-1}\E_{X_0,\eta_0}\sup_{0\leq s\leq t} \int^1_0
\Big|\sum_{j=1}^\infty \int_0^s  \Big( \Sigma^j(x,  \eta^\phi_k(u)|_0^1,  X_k^\phi(u,x), Z_k^\phi(u) )\\
&\qquad\qquad \qquad\qquad \qquad\qquad \qquad\qquad 
- \Sigma^j(x,  \eta^\phi_{k-1}(u)|_0^1,  X_{k-1}^\phi(u,x), Z_{k-1}^\phi(u)  )\Big)\, \di B^j(u)\Big|^p\,\di x.
\ea
Applying Lemma~\ref{l:int} to the term containing stochastic integrals, the Jensen inequality to the Lebesgue integral, the  
Burkholder inequality to the stochastic integral and finally the Lipschitz property of $F$ and $\|\Sigma\|_{l^2}$ we get 
\ba
\label{estimateF1}
\E_{X_0,\eta_0}\sup_{0\leq s\leq t} \int^1_0 |  X^\phi_{k+1}(s,x) -  X^\phi_k (s,x)|^p\, \di x
&\leq C_1 \int_0^t \E_{X_0,\eta_0} \sup_{0\leq u\leq s}|\eta^\phi_k(u)-  \eta^\phi_{k-1}(u)|^p\,\di s\\
&+C_1\int_0^t \E_{X_0,\eta_0}\sup_{0\leq u\leq s}\int_0^1| X_k^\phi(u,x)- X_{k-1}^\phi(u,x)|^p\,\di x\,\di s\\
&+C_1 \int_0^t \E_{X_0,\eta_0} \sup_{0\leq u\leq s}| Z_k^\phi(s) - Z_{k-1}^\phi(s)|^p \,\di s .
\ea
To estimate the term with $Z^\phi$ we take into account \eqref{e:est} to get for $0\leq s\leq t\leq T$
\ba
\label{estimateXbar}
| Z_k^\phi(s) - Z_{k-1}^\phi(s)|^p
&=\abs{\left(  X_k (s,0)    -  X_{k-1} (s,0)  \right)\mu ([0,1])+ \int^1_0 \left(  \phi ( Y_k ^\phi(s,z))  -  \phi ( Y_{k-1} ^\phi(s,z))\right)  \mu ([z,1])
\,\di z}^p\\
&\leq 2^{p-1}M^p |X_k (s,0)-X_{k-1}(s,0)|^p + 2^{p-1}M^p \int_0^1 |Y_k^\phi(s,x)- Y_{k-1}^\phi(s,x)|^p\,\di x.
\ea
Therefore we can conclude that
\ba
\label{estimateC_1}
\E_{X_0,\eta_0}\sup_{0\leq s\leq t}\int_0^1 |  X^\phi_{k+1}(s,x) -  X^\phi_k (s,x)|^p\,\di x \leq C_2 \int^t_0 G_k(s)\,\di s.
\ea
2. Analogously we obtain a similar estimate for $\E_{X_0,\eta_0} \sup_{0\leq s\leq t}|X^\phi_{k+1}(t,0) -  X^\phi_k (t,0)|^p$ with some constant $C_3>0$.

\noindent
3. The fourth summand of \eqref{bigsystem} is estimated similarly.
The terms containing $F_x$ and $\|\Sigma_x\|_{l^2}$  
are estimated in the same way as above. 
For the terms containing  the product $F_X\cdot\phi(Y^\phi)$ 
we write 
\ba
 &\Big| F_X (x,\eta^\phi_k(s)|_0^1, X_k^\phi(s,x), Z_k^\phi(s))\phi(Y_{k}^\phi(s,x) )
 - F_X (x,\eta^\phi_{k-1}(s)|_0^1, X_{k-1}^\phi(s,x), Z_{k-1}^\phi(s) )\phi ( Y_{k-1}^\phi(s,x) )\Big|^p\\
&\leq \Big| F_X (x,\eta^\phi_k(s)|_0^1, X_k^\phi(s,x), Z_k^\phi(s) )\phi ( Y_{k}^\phi(s,x) )
- F_X (x,\eta^\phi_{k}(s)|_0^1, X_{k}^\phi(s,x), Z_k^\phi(s))\phi ( Y_{k-1}^\phi(s,x) )\Big|^p\\
&+ \Big| F_X (x,  \eta^\phi_{k}(s)|_0^1, X_{k}^\phi(s,x), Z_k^\phi(s) )\phi ( Y_{k-1}^\phi(s,x) )
- F_X (x,  \eta^\phi_{k-1}(s)|_0^1, X_{k-1}^\phi(s,x), Z_{k-1}^\phi(s))\phi ( Y_{k-1}^\phi(s,x) )\Big|^p\\
&\leq \| F_X\|_\infty^p\cdot |\phi ( Y_{k}^\phi(s,x) )-\phi ( Y_{k-1}^\phi(s,x) )|^p\\
&+C_3 \Big|\phi ( Y_{k-1}^\phi(s,x) )\Big|^p 
\Big(\Big| \eta^\phi_k(s)|_0^1- \eta^\phi_{k-1}(s)_0^1\Big|^p+| X_k^\phi(s,x)- X_{k-1}^\phi(s,x)|^p+| Z_k^\phi(s)  - Z_{k-1}^\phi(s)|^p \Big)\\
&\leq C_4 | Y_{k}^\phi(s,x) - Y_{k-1}^\phi(s,x)|^p \\
&\ +C_4 \Big[| \eta^\phi_k(s)- \eta^\phi_{k-1}(s)|^p
+| X_k^\phi(s,x)- X_{k-1}^\phi(s,x)|^p+| Z_k^\phi(s)  - Z_{k-1}^\phi(s)|^p \Big],
\ea
where $C_4$ depends on $C_\phi$.
We have used here that $\phi$ is bounded by $C_{\phi} $ and its Lipschitz constant is 1. The term containing $\|\Sigma_X\|_{l^2}\cdot \phi(Y^\phi)$
is estimated analogously, and hence we obtain that
\ba
\label{estimateC_3}
\E_{X_0,\eta_0}&\sup_{0\leq s\leq t}\int_0^1| Y^\phi_{k+1}(s,x)-Y^\phi_k (s,x)|^p\,\di x \leq C_5\int^t_0 G_k(s)\,\di s,
\ea
where $C_5$ depends on $C_\phi$.

\noindent
4.
To estimate the first term in \eqref{bigsystem} we follow the above arguments to get
\ba
\E_{X_0,\eta_0}\sup_{0\leq s\leq t}|  \eta^\phi_{k+1}(s) -  \eta^\phi_{k}(s)|^p
&\leq C_6\E_{X_0,\eta_0}\Big[|\eta^\phi_k(s)|^1_0-  \eta^\phi_{k-1}(s)|^1_0|^p
+C_6| X^\phi_k(s,0)  -  X^\phi_{k-1}(s,0)|^p \\
&+C_6 \Big| \int_0^{\eta^\phi_k(s)|^1_0} \phi ( Y_k^\phi(s,z))\, \di z-\int_0^{\eta^\phi_{k-1}(s)|^1_0} \phi ( Y_{k-1}^\phi(s,z))\, \di z\Big|^p\\
&+|Z_k^\phi(s)- Z_{k-1}^\phi(s)|^p\Big].
\ea
Furthermore
we get
\ba
\Big|\int _0^{\eta^\phi_k(s)|^1_0}& \phi ( Y_k^\phi(s,z))\, \di z-\int _0^{  \eta^\phi_{k-1}(s)|^1_0} \phi ( Y_{k-1}^\phi(s,z))\,\di z\Big|^p\\
&\leq 2^{p-1}\Big|\int _0^{\eta^\phi_k(s)|^1_0} \Big(\phi ( Y_k^\phi(s,z))- \phi ( Y_{k-1}^\phi(s,z))\Big)\, \di z\Big|^p 
+2^{p-1}\Big| \int_{  \eta^\phi_{k-1}(s)|^1_0}^{  \eta^\phi_k(s)|^1_0} \phi ( Y_{k-1}^\phi(s,z))\, \di z\Big|^p\\
&\leq 2^{p-1} \int _0^1  | Y_k^\phi(s,z) -   Y_{k-1}^\phi(s,z)|^p\, \di z +2^{p-1}C_{\phi}^p \big| \eta^\phi_{k-1}(s)|^1_0-  \eta^\phi_k(s)|^1_0\big|^p\\
&\leq 2^{p-1}   \int_0^1| Y_k^\phi(s,x) -   Y_{k-1}^\phi(s,x)|^p\,\di x 
+2^{p-1}C_{\phi}^p | \eta^\phi_{k-1}(s) -  \eta^\phi_k(s) |^p.
\ea
This gives us for a certain $C_4>0$ depending on $\phi$ that
\ba
\label{estimateC_4}
\E_{X_0,\eta_0}\sup_{0\leq s\leq t}\abs{  \eta^\phi_{k+1}(s) -  \eta^\phi_{k}(s) }^p\leq C_4 \int^t_0 G_k(s)\,\di s.
\ea
Combining the above estimates we find that
\ba
G_{k+1}(t)\leq C_5  \int_0^t G_k(s)\, \di s ,
\ea
for some constant $C_5>0$ that depends on $\phi$ and $T$.
Since
\ba
G_0(t):= \int_0^1 | X_{0}(x)|^p\,\di x +|X_0(0)|^p+  \int_0^1 | \nabla X_{0}(x)|^p\,\di x    +|\eta_0|^p:= C_0<\infty
\ea
it is easy to verify inductively that
\ba
\sum_{k=0}^\infty G_{k}(T)<\infty. 
\ea
Therefore there exist limit $L^p$-integrable processes $X^\phi$, $Y^\phi$ and a limit process $\eta^\phi$.
It is easy to verify that they satisfy equation \eqref{diffeqtrun}. 
The solutions are continuous by Kolmogorov's continuity criterion. We omit the details here.
\end{proof}

\begin{lem} 
\label{W1grad}
For any $T>0$, $p\geq 2$ and $\phi$ satisfying \eqref{e:phi}
\begin{align}
\label{e:estYphi}
&\E_{X_0,\eta_0}\, \sup_ {t\leq T}\int^1_0| Y^\phi(t,x) |^p \, \di x \leq C_{\nabla}(p,T)\Big(1+ \int_0^1 |\nabla X_0(x)|^p\,\di x\Big),\\
\label{e:estNablaXphi}
&\E_{X_0,\eta_0}\, \sup_ {t\leq T}\int^1_0|\nabla  X^\phi(t,x)|^p\, \di x  \leq C_{\nabla}(p,T)\Big(1+ \int_0^1 |\nabla X_0(x)|^p\,\di x\Big),
\end{align}
where the constant $C_{\nabla}(p,T)>0$ does not depend on $C_\phi$.
\end{lem}
\begin{proof}
For $p\geq 2$ and $t\in[0,T]$ we recall the inequality \eqref{e:Y} that holds for $Y^\phi$ as a limit of $\{Y_k\}$ as well as for $\nabla X$. In the 
latter case one may have to apply the localization argument together with the Fatou lemma. We note that the constant $C_3$ in \eqref{e:Y}
does not depend on $C_\phi$.
%
\end{proof}

\begin{lem}
\label{l:nablaY}
For any $T>0$, $p\geq 2$ and any $\phi$ satisfying \eqref{e:phi} there is a constant $C_{\nabla Y}(p,T)>0$ such that
\ba
\label{e:nablaYp}
\E_{\eta_0, X_0 } &\,\sup_ {t\leq T}\int_0^1 |\nabla Y^{\phi}(t,x)|^p \,\di x 
\leq C_{\nabla Y}(p,T)\Big(1+ \int_0^1 |\nabla^2 X_0(x)|^p\,\di x 
+\int_0^1 |\nabla X_0(x)|^{2p}\,\di x  \Big) .
\ea
where the constant $C_{\nabla Y}(p,T)>0$ does not depend on $C_\phi$.
\end{lem}
\begin{proof}
 The gradient $ \nabla Y^\phi(t,x)$ satisfies the differential equation
\ba
\nabla  Y^{\phi}(t,x)= \nabla^2 X_0(x)+ \int ^t_0 &\Big[ F_{xx}+ F_{xX}\cdot \nabla  X^{\phi}(s,x)+ F_{xX}\cdot\phi( Y^{\phi}(s,x))\\
&+  F_{XX}\cdot \nabla  X^{\phi}(s,x)\cdot \phi( Y^{\phi}(s,x))
+  F_{X}\cdot  \phi^\prime( Y^\phi(s,x))\cdot \nabla  Y^{\phi}(s,x)\Big]\,\di s\\
+ \sum_{j=1}^\infty\int^t_0 & \Big[\Sigma^j_{xx}+ \Sigma^j_{xX}\cdot  \nabla  X^{\phi}(s,x)+\Sigma^j_{xX}\cdot \phi( Y^{\phi}(s,x))\\
&+ \Sigma^j_{XX}\nabla X^{\phi}(s,x)\cdot \phi( Y^{\phi}(s,x))+  \Sigma^j_{X} \cdot \phi'( Y^{\phi}(s,x))\cdot \nabla  Y^{\phi}(s,x)\Big]\,\di B^j(s).
\ea
Repeating the argument of Lemma \ref{l:boun}, the bounds of Lemma \ref{W1grad}, and 
the fact that $|\phi(x)|\leq |x|$ and $|\phi'(x)|\leq 1$ we get
\ba
\E_{X_0,\eta_0} &\,\sup_ {s\leq t}\int^1_0 |\nabla  Y^{\phi}(s,x)|^p\, \di x\\
&\leq C_1\int^1_0 |\nabla ^2 X_0(x)|^p\,\di x+ C_2+ C_3\int^t_0\E_{X_0,\eta_0}\int^1_0|\nabla  X^\phi(s,x)|^p\, \di x\,\di s\\
&+  C_4\int^t_0\E_{X_0,\eta_0}\int^1_0 |\phi( Y^\phi(s,x))|^p\,\di x\,\di s
+C_5\int^t_0\E_{X_0,\eta_0}\int^1_0 |\nabla  X^\phi(s,x)\phi( Y^\phi(s,x))|^p\, \di x\,\di s\\
&+C_6\int^t_0 \E_{X_0,\eta_0} \int^1_0  |\phi'( Y^\phi(s,x))|^p|\nabla  Y^\phi(s,x)|^p\, \di x\,\di s \\
&\leq 
C_1\int^1_0 |\nabla ^2 X_0(x)|^p\,\di x+ C_2+ C_3\int^t_0\E_{X_0,\eta_0}\int^1_0|\nabla  X^\phi(s,x)|^p\, \di x\,\di s\\
&+  C_4\int^t_0\E_{X_0,\eta_0}\int^1_0 |Y^\phi(s,x)|^p\,\di x\,\di s
+C_5\int^t_0\E_{X_0,\eta_0}\int^1_0 |\nabla  X^\phi(s,x)|^p|Y^\phi(s,x)|^p\, \di x\,\di s\\
&+C_6\int^t_0 \E_{X_0,\eta_0} \int^1_0 |\nabla  Y^\phi(s,x)|^p\, \di x\,\di s \\
&\leq 
C_1\int^1_0 |\nabla ^2 X_0(x)|^p\,\di x+ C_2+ C_7\Big(1+\int^1_0|\nabla  X_0(x)|^p\, \di x\Big)\\
&+C_8 \Big(1+\int^1_0|\nabla  X_0(x)|^{2p}\, \di x\Big)
+C_6\int^t_0 \E_{X_0,\eta_0} \sup_ {u\leq s}\int^1_0 |\nabla  Y^\phi(u,x)|^p\, \di x\,\di s .
%
\ea
Using the elementary inequality
\ba
\int^1_0|\nabla  X_0(x)|^p\, \di x\leq \int^1_0 (|\nabla  X_0(x)|+1)^p\, \di x
\leq  \int^1_0 (|\nabla  X_0(x)|+1)^{2p}\, \di x\leq   \int^1_0 (|\nabla  X_0(x)|+1)^{2p}\, \di x
\ea
and applying Gronwall's inequality yields the result.
\end{proof}

\subsection{Proof of Theorem \ref{existence}}

For $n\in\mathbb N$, let
\ba
\label{e:phi-exa}
\phi_n(x)&=\begin{cases}
\displaystyle           x,\quad x\in[0,n],\\
\displaystyle           n+ \int_0^{x-n} \Big( 1-\frac{2 \arctan(y^3)}{\pi}\Big)\,\di y,\quad x\in(n,+\infty),\\
          \end{cases}\\
\phi_n(x)&=-\phi_n(-x),\quad x<0.          
\ea
It is clear that each $\phi_n$, $n\in\mathbb N$, satisfies conditions \eqref{e:phi} with $C_{\phi_n}=n+2/\sqrt 3$.
Let 
\ba
\tau_n=\inf_{t\geq 0}\Big\{ \sup_ {x \in [0,1]}|Y^{\phi_n}(t,x)|\geq n\Big\} \wedge T.
\ea
For all $t\in[0, \tau_n]$ we have $\phi_n ( Y^{\phi_n}(t,x) )=  Y^{\phi_n}(t,x)$ and therefore for $t\in[0, \tau_n]$
\ba
\eta(t)&=\eta^{\phi_n}(t),\\
X(t,x)&=X^{\phi_{n}}(t,x),\\
\nabla X(t,x)&=Y^{\phi_{n}}(t,x).
\ea
Moreover, on $[0, \tau_n]$ we have $Y^{\phi_n}=Y^{\phi_{m}}$ for all $m>n$.

Let us show that $\tau_n\to T$ in probability as $n\to\infty$.
Indeed, for any $T>0$ and $p\geq 2$, 
due to Lemmas \ref{W1grad} and \ref{l:nablaY} there is a constant $C=C(p,T,X_0)>0$ such that for all $n\in\mathbb N$
\ba
\label{W1pnormY}
 \E_{X_0,\eta_0} \sup_ {t\leq T} \| Y^{\phi_n}(t,\cdot)\|_{W^{1,p}}\leq C .
\ea
Hence by Markov's inequality we get
 \ba
\P_{X_0,\eta_0}(\tau_n\leq T)
&= \P_{X_0,\eta_0}\Big(\sup_ {t\leq T} \sup_ {x \in [0,1]}| Y^{\phi_n}(t,x)|\geq n\Big)\\
&\leq \frac{1}{ n} \E_{X_0,\eta_0} \sup_ {t\leq T} \sup_ {x \in [0,1]} | Y^{\phi_n}(t,x)|\\
&\leq  \frac{C_\text{Sob}(1,p)}{ n}\E_{X_0,\eta_0}\,\sup_ {t\leq T} \| Y^{\phi_n}(t,\cdot)\|_{W^{1,p}}
\to 0,\quad n\to\infty.
\ea

\section{Proof of Theorem \ref{etain01}\label{s:01}}

Let $(X,\eta)$ be the solution of \eqref{diffeq} with the initial values $X_0\in C^2([0,1],\bR)$ and $\eta_0\in (0,1)$.
We show that $\eta$ does not hit $0$ or $1$ on $[0,T]$. 

Let 
\ba
\rho(t):=-\ln\eta(t) \in [0,+\infty].
\ea
We show that $\rho(t)<+\infty$ a.s.\ on $[0,T]$.
Let for $N\geq 1$
\ba
\sigma_N:=\inf\{t\geq 0\colon \rho(t)>N\}\wedge T.
\ea
By the It\^o formula, 
\ba
\rho(t\wedge \sigma_N)=-\ln \eta_0 
&-\int_0^{t\wedge\sigma_N} \frac{f(\eta(s), X(s,\eta(s)),Z(s))}{\eta(s)}   \, \di s \\
&- \int_0^{t\wedge \sigma_N}\frac{\sigma(\eta(s), X(s,\eta(s)),Z(s))}{\eta(s)}  \,  \di W(s)\\
&+\frac12\int_0^{t\wedge \sigma_N}\frac{\sigma(\eta(s), X(s,\eta(s)),Z(s))^2}{\eta(s)^2}\,   \di s .
\ea
Let $C>0$ denote the Lipschitz constant for $f$ and $\sigma$.
By assumption $\mathbf{A}_4$, $f(0,X,Z)\geq 0$ and $\sigma(0,X,Z)=0$, hence
\ba
\label{e:f01}
-\frac{f(\eta,X,Z)}{\eta}&=-\frac{f(0,X,Z)}{\eta} + \frac{f(0,X,Z) - f(\eta,X,Z)}{\eta}
\leq -\frac{f(0,X,Z)}{\eta} + C\leq C
\ea
and
\ba
\label{e:s01}
\Big|\frac{\sigma(\eta, X,Z)}{\eta}\Big|&=\Big|\frac{\sigma(\eta, X,Z)-\sigma(0, X,Z) }{\eta}\Big|\leq C.
\ea
Consequently, with the help of the Burkholder inequality we get
\ba
\E_{X_0,\eta_0} \sup_{t\leq T}\rho(t\wedge\sigma_N)\leq -\ln \eta_0 + \Big(C+\frac12C^2\Big)T +C\sqrt T=:C(T,\eta_0), 
\ea
and by Fatou's Lemma
\ba
\E_{X_0,\eta_0} \sup_{t\leq T}\rho(t)\leq C(T,\eta_0).
\ea
Therefore, $\eta(t)>0$ a.s.\ on $[0,T]$.
The bound $\eta(t)<1$ a.s.\ is obtained analogously.

\section{Proof of Theorem~\ref{wnormmomentsboundthm}\label{s:stability}}

The estimate \eqref{e:estX1} of Theorem~\ref{wnormmomentsboundthm} will follow from  
Lemmas \ref{nablamomentsboundlem} and Lemma \ref{momentsboundlem}.
The estimate \eqref{e:estnablaX} of Theorem~\ref{wnormmomentsboundthm} will follow from  
Lemmas \ref{nablamomentsboundlem} and \ref{nabla2momentsboundlem}.

\begin{lem}
\label{nablamomentsboundlem}
Let Assumptions $\mathbf{A}_1$--$\mathbf{A}_5$ hold true, $\eta_0\in (0,1)$ and $X_0\in C^2([0,1],\bR)$.
Then there is $p'>4$ such that for any $p\in[2,p']$ there is a constant $C_{\nabla X}(p)>0$ such that for all $\e\in(0,1]$
\ba
\label{EnablaXepsbound}
 \sup_ {t\geq 0} 
 \E_{X_0,\eta_0} \|\nabla X^\e(t,\cdot)\|^{p}_{L^p} \leq C_{\nabla X}(p) \Big(1+  \|\nabla X_0(\cdot)\|^{p}_{L^p}\Big).
\ea
\end{lem}

\begin{proof}
Let $\delta>0$ be the constant from  Assumption $\mathbf{A}_6$.
Let $p'=4+\frac{\delta}{\|\Sigma_X\|_{l^2}^2+1}$.
Due to Assumption $\mathbf{A}_6$, for $p\in[2,p']$ we have
\ba
\label{e:estF}
F_X &+  \frac{p-1}{2} \|\Sigma_X\|_{l^2}^2
= F_X+\frac32 \|\Sigma_X\|_{l^2}^2 + \frac{p-4}{2}\|\Sigma_X\|_{l^2}^2
\leq -\frac{\delta}{2}<0. 
\ea
Recall that the process $\nabla  X^\e$ is a solution of the variational equation 
\ba
\label{diffepnablaXeps}
\nabla  X^\e (t,x)=\nabla X_0(x)
 +\frac{1}{\e}\int_0^t \Big(  F_x   +  F_ X\nabla  X^\e (s,x)  \Big)  \,\di s 
 +\frac{1}{\sqrt{\e}}\sum_{j=1}^\infty \int_0^t \Big( \Sigma^j_x  +   \Sigma^j_X \cdot \nabla  X^\e (s,x) \Big) \, \di B^j(s).
\ea
We apply the It\^o formula to the function $x\mapsto |x|^p$, which is twice continuously differentiable for $p\geq 2$. We get
\ba
\label{e:nX}
|\nabla X^\e(t,x)|^{p}
&=|\nabla X_0(x)|^{p}
+  \frac{p}{\e}\int^t_0 \Big(F_x \cdot|\nabla X^\e(s,x)|^{p-1} \sgn(\nabla X^\e(s,x)) + F_ X \cdot |\nabla X^\e(s,x)|^p\Big)  \,\di s \\
&+\frac{p}{\sqrt{\e}}\sum_{j=1}^\infty
\int_0^t\Big(\Sigma^j_x\cdot |\nabla X^\e(s,x)|^{p-1} \sgn(\nabla X^\e(s,x)) + \Sigma^j_X \cdot |\nabla X^\e(s,x)|^p \Big)  \, \di B^j(s) \\
&+\frac {p(p-1)} { 2\e}\int_0^t |\nabla X^\e(s,x)|^{p-2}\sum_{j=1}^\infty \Big( \Sigma^j_x  + \Sigma^j_X  \nabla  X^ \e (s,x)   \Big)^2\, \di s.
\ea
Furthermore,
\ba
\sum_{j=1}^\infty \Big( \Sigma^j_x  + \Sigma^j_X  \nabla  X^ \e (s,x)   \Big)^2
\leq  \|\Sigma_x\|^2_{l^2} 
+ |\nabla  X^ \e (s,x)|^2  \|\Sigma_X\|^2_{l^2}
+2 |\nabla  X^ \e (s,x)|\cdot  \|\Sigma_x|_{l^2} \cdot  \| \Sigma_X\|_{l^2} .
\ea
Therefore due to boundendness of all the derivatives, there are $\delta_1>0$ and $C_1=C_1(p)>0$ such that for all $y\in\bR$
\ba
p&\Big(F_x \cdot|y|^{p-2} y + F_ X \cdot |y|^p\Big)+\frac {p(p-1)} { 2}|y|^{p-2}\sum_{j=1}^\infty \Big( \Sigma^j_x  + \Sigma^j_X  y \Big)^2\\
&\leq p \Big(F_ X + \frac{p-1}{2}\|\Sigma_X\|^2_{l^2}\Big)|y|^p 
+ p\Big(F_x \cdot|y|^{p-2} y +\frac{p-1}{2} \|\Sigma_x\|^2_{l^2} |y|^{p-2} + (p-1)  \|\Sigma_x|_{l^2} \cdot  \| \Sigma_X\|_{l^2}\cdot |y|^{p-1}  \Big)\\
&\leq  C_1- \delta_1|y|^p.
\ea
Taking expectation in \eqref{e:nX}, integrating w.r.t.\ $x\in [0,1]$, and applying Fatou's lemma localization argument we get
\ba
\E_{X_0,\eta_0} \int_0^1 |\nabla X^\e(t,x)|^p\,\di x 
&\leq \int_0^1 |\nabla X_0(x)|^p\,\di x
 +  \frac{1}{\e} \int_0^t \Big(C_1- \delta_1 \E_{\eta_0,X_0(\cdot)} \int_0^1 |\nabla X^\e(s,x)|^p\,\di x\Big)\,\di s.
\ea
Therefore,
\ba
\E_{X_0,\eta_0} \int_0^1 |\nabla X^\e(t,x)|^p\,\di x 
&\leq \ex^{-\delta_1 t/\e} \int_0^1|\nabla X_0(x)|^p\,\di x + \frac{C_1}{\e}\int_0^t \ex^{-\delta_1(t-s)/\e}\,\di s\\
&\leq  \int_0^1|\nabla X_0(x)|^p\,\di x + \frac{C_1}{\delta_1}.
\ea
\end{proof}

\begin{lem}
\label{momentsboundlem}
Under the sets of Assumptions  $\mathbf{A}_1$--$\mathbf{A}_7$, for any $\eta_0\in (0,1)$ and $X_0\in C^2([0,1],\bR)$ 
there is $C_{X} (1,2)>0$ such that for all $\e\in(0,1] $
\ba
\label{EXepsbound}
\sup_ {t\geq 0}\E_{X_0,\eta_0}\|X^\e(t,\cdot)\|^2_{L^2} \leq C_{X}(1,2) \Big(1+ \|X_0(\cdot)\|_{L^2}^2 
+  \|\nabla X_0(\cdot)\|_{L^2}^2 
\Big).
\ea
\end{lem}
\begin{proof}
Due to Assumption $\mathbf{A}_7$ we have that
\ba
X F(x,\eta,X,Z)\leq -AX^2 +BXZ+C
\ea
where
\ba
A>\frac{B}{2}(1+C_\mu^2).
\ea
By the It\^o formula we get
\ba
|X^\e(t,x)|^2= |X_0(x)|^2 &+ \frac{2}{\e}\int_0^t F\cdot X^\e(s,x) \,\di s 
+\frac{1}{\e}\sum_{j=1}^\infty \int_0^t  |\Sigma^j|^2\,\di s \\
&+ \frac{1}{\sqrt\e}\sum_{j=1}^\infty \int_0^t \Sigma^j\cdot  X^\e(s,x)\,\di B^j(s),
\ea
and hence
\ba
\E_{X_0,\eta_0}  |X^\e(t,x)|^2&
\leq |X_0(x)|^2 
+ \frac{2}{\e}\E_{X_0,\eta_0} \int_0^t  \Big(-A |X^\e(s,x)|^2 +B X^\e(s,x) Z^\e(s)+C + \frac12 \|\Sigma\|^2_{l^2}\Big)\,\di s .
\ea
Using the Cauchy-Schwarz inequality and the elementary inequality $xy\leq \frac12(x^2+y^2)$ we get
\ba
\E_{X_0,\eta_0} |X^\e(t,x)|^2&\leq |X_0(x)|^2\\
&+ \frac{2}{\e}\E_{X_0,\eta_0}\int_0^t \Big[\Big(-A+\frac{B}{2}\Big) |X^\e(s,x)|^{2} + \frac{B }{2} |Z^\e(s)|^2
+C+\frac 12 \|\Sigma\|^2_{l^2}\Big]\,\di s .
\ea
Integrating the last inequality w.r.t.\ $x$ and taking into account \eqref{e:ZZZ} and Lemma \ref{nablamomentsboundlem} yields
\ba
&\E_{X_0,\eta_0}\|X^\e(t,x)\|^2_{L^2} \\
&\leq \|X_0(\cdot)\|^2_{L^2}
+ \frac{2}{\e}\E_{\eta_0,X_0(\cdot)}\int_0^t \Big[\Big(  -A+\frac{B}{2}+\frac{B }{2}C_\mu^2\Big)\|X^\e(s,x)\|_{L^2}^2 
+C + \frac12\|\Sigma\|^2_{l^2}
+\frac{B}{2} C_\mu^2 \|\nabla X(s,\cdot)\|^{2}_{L^2}\Big)
\Big]\,\di s \\
&\leq 
\|X_0(\cdot)\|^2_{L^2}
+ \frac{2}{\e}\int_0^t \Big[-\delta_1\E_{X_0,\eta_0}\|X^\e(s,x)\|_{L^2}^2
+ C_1\Big(1+  \|\nabla X_0(\cdot)\|^{2}_{L^2}\Big)
\Big]\,\di s 
\ea
with some positive constants $\delta_1$ and $C_1$.
The statement follows in the same way as in the previous lemma.
\end{proof}

\begin{lem}
\label{nabla2momentsboundlem}
 
Under the sets of assumptions $\mathbf{A}_1$--$\mathbf{A}_7$ there is $p^*> 2$ such that for all $p\in[2,p^*)$ there is a $C_{\nabla^2 X}(p^*)>0$
such that for all $\e\in(0,1]$
\ba
 \sup_ {t\geq 0} \E_{X_0,\eta_0}\|\nabla^2 X^\e(t,x)\|_{L^p}^p 
 \leq C_{\nabla^2 X}(p^*) \Big(1+\|\nabla^2 X_0(\cdot)\|_{L^p}^p +\|\nabla X_0(\cdot)\|_{L^{2p^*}}^{2p^*}  \Big).
\ea
\end{lem}
\begin {proof}
We choose $p^*=\frac{p'}{2}\wedge \frac52$, where $p'$ has been defined in the proof of Lemma \ref{momentsboundlem}.
 
The second derivative $\nabla^2  X^\e$ satisfies the stochastic differential equation
\ba
\nabla^2  X^\e (t,x)&= \nabla^2 X_0(x)+ \frac{1}{\e}\int^t_0 \Big[ F_{xx}+ 2F_{xX} \nabla  X^\e(s,x)
+  F_{XX}|\nabla  X^\e(s,x)|^2+  F_{X} \nabla^2 X^\e(s,x)\Big]\,\di s\\
& + \frac{1}{\sqrt\e}\sum_{j=1}^\infty\int^t_0 \Big[\Sigma^j_{xx}+ 2\Sigma^j_{xX} \nabla  X^\e(s,x)
+ \Sigma^j_{XX}\cdot |\nabla  X^\e(s,x)|^2+  \Sigma^j_{X} \nabla^2  X^\e(s,x)\Big]\,\di B^j(s).\\
\ea
Applying the It\^o formula for $p\geq 2$ we get
\ba
\label{e:diff2X}
|\nabla^2  X^\e (t,x)|^p&= |\nabla^2 X_0(x)|^p\\
&+ \frac{p}{\e}\int^t_0 |\nabla^2  X^\e |^{p-2}\nabla^2  X^\e \Big[ F_{xx}+ 2F_{xX} \nabla  X^\e 
+  F_{XX}|\nabla  X^\e|^2+  F_{X} \nabla^2 X^\e \Big]\,\di s\\
& + \frac{p(p-1)}{2\e}\sum_{j=1}^\infty\int^t_0   |\nabla^2  X^\e  |^{p-2}  \Big[\Sigma^j_{xx}+ 2\Sigma^j_{xX} \nabla  X^\e 
+ \Sigma^j_{XX}\cdot |\nabla  X^\e|^2+  \Sigma^j_{X} \nabla^2  X^\e \Big]^2\,\di s\\
&+\text{local martingale}\\
&=  \frac{p}{\e}\int^t_0 I^\e(s)  \,\di s +\text{local martingale}.
\ea
We use
the elementary inequalities $(a+b)^2\leq 2(a^2+b^2)$ and $(a+b+c)^2\leq 3(a^2+b^2+c^2)$,
and collect the terms in the integrands of \eqref{e:diff2X} to get the estimate
\ba
I^\e
&\leq |\nabla^2  X^\e |^{p-2}\nabla^2  X^\e \Big[ F_{xx}+ 2F_{xX} \nabla  X^\e 
+  F_{XX}(\nabla  X^\e)^2+  F_{X} \nabla^2 X^\e\Big]\\
&+(p-1)  |\nabla^2  X^\e|^p \| \Sigma_{X}\|_{l^2}^2
+ 3(p-1) |\nabla^2  X^\e|^{p-2} \Big[  \|\Sigma_{xx}\|^2_{l^2}+ 4\| \Sigma_{xX}\|_{l^2}^2  |\nabla  X^\e|^2 
+ \|\Sigma_{XX}\|^2_{l^2} \cdot |\nabla  X^\e |^2\Big]\\
& \leq |\nabla^2  X^\e |^{p} F_X + (p-1)  |\nabla^2  X^\e|^p \| \Sigma_{X}\|_{l^2}^2\\
&\qquad +C_1\Big( |\nabla^2  X^\e |^{p-1} +  |\nabla^2  X^\e |^{p-1}| \nabla  X^\e| 
+ |\nabla^2  X^\e |^{p-1} |\nabla  X^\e|^2
+ |\nabla^2  X^\e|^{p-2} \\
&\qquad + |\nabla^2  X^\e|^{p-2} |\nabla  X^\e|^2 
+|\nabla^2  X^\e|^{p-2} |\nabla  X^\e |^4\Big),
\ea
where $C_1>0$ is the universal bound for all the bounded derivatives $F_{xx}$, $2F_{xX}$ etc.

To estimate the products on the last formula we apply
the Young inequality and estimate
\ba
&|\nabla^2 X^\e|^{p-1} \cdot \Big(|\nabla X^\e|+ |\nabla X^\e|^2\Big)
\leq \frac{p''-1}{p''}|\nabla^2 X^\e|^{\frac{p^*(p-1)}{p^*-1}}  + \frac{1}{p^*}\Big(|\nabla X^\e|+ |\nabla X^\e|^2\Big)^{p^*},\\
&|\nabla^2 X^\e|^{p-2} \cdot  \Big(|\nabla X^\e|^4 + |\nabla X^\e|^2\Big)
\leq \frac{p^*-2}{p^*}|\nabla^2 X^\e|^{\frac{p^*(p-2)}{p^*-2}}  + \frac{2}{p^*} \Big(|\nabla X^\e|^4 + |\nabla X^\e|^2\Big)^{p^*/2}.
\ea
Note that for $p<p^*$
\ba
\frac{p^*(p-1)}{p^*-1}<p\quad \text{ and }\quad \frac{p^*(p-2)}{p^*-2}<p,
\ea
so that we conclude that there are $C_2,C_3>0$ such that 
\ba
I^\e
&\leq |\nabla^2  X^\e |^{p} \Big(F_X + (p-1) \| \Sigma_{X}\|_{l^2}^2\Big)
+ C_2 +C_3 |\nabla X^\e|^{2 p^*}.
\ea
We recall \eqref{e:estF} and see that
for all $p\in[2,p^*)$ and some $\delta_1>0$
\ba
I^\e\leq  -\delta_1 |\nabla^2  X^\e |^p + C_2 + C_3 |\nabla X^\e|^{2p^*}.
\ea
Integrating w.r.t.\ $x$, taking expectation and using \eqref{EnablaXepsbound}
we get for some positive $C_4$ and $C_5$ that
\ba
\E_{X_0,\eta_0}\|\nabla^2  X^\e (t,\cdot)\|_{L^p}^p
&\leq \|\nabla^2  X_0 (\cdot)\|_{L^p}^p \\
&\qquad +\frac{p}{\e}\int_0^t 
\Big(-\delta_1 \E_{X_0,\eta_0}\|\nabla^2 X^\e(s,\cdot)\|^p_{L^p} + C_2
+ C_3 \E_{X_0,\eta_0}\|\nabla  X^\e(s,\cdot)\|^{2p^*}_{L^{2p^*}}\Big)\,\di s\\
&\leq \|\nabla^2  X_0 (\cdot)\|_{L^p}^p\\
&\qquad +
\frac{p}{\e}\int_0^t \Big(-\delta_1 \E_{X_0,\eta_0}\|\nabla^2 X^\e(s,\cdot)\|^p_{L^p} 
+C_4 + C_5 \|\nabla  X_0(\cdot)\|^{2p^*}_{L^{2p^*}}\Big)\,\di s,
\ea
and the statement follows.
\end{proof}

\section{Proof of convergence\label{s:convergence}}

In this Section we always assume that Assumptions $\mathbf{A}_1$--$\mathbf{A}_7$ hold true, $\eta_0\in (0,1)$ and $X_0(\cdot)\in C^2([0,1],\bR)$.

\subsection {Auxiliary results}

\begin{lem}
\label{etaboundthm}
For any $p\geq 2$ and $\Delta\in(0,1]$ there is a constant $C_\eta(p)>0$ such that for all $\e\in(0,1]$
\ba
\label{etabound}
\sup_ {t\geq 0,0\leq s\leq \Delta}\E_{X_0,\eta_0}|\eta^\e(t+s)- \eta^\e(t)|^p 
\leq C_\eta (p) \Delta^\frac{p}{2}. 
\ea
\end{lem}
\begin{proof}
For $s,t\geq 0$, by the Markov property we have
\ba
\E_{X_0,\eta_0}|\eta^\e(t+s)- \eta^\e(t)|^p
&=\E_{X_0,\eta_0}\Big[\E\Big[|\eta^\e(t+s)- \eta^\e(t)|^p\Big|\rF_t^\e\Big]\Big]\\
&=\E_{X_0,\eta_0}\Big[\E_{X^\e(t,\cdot),\eta^\e(t)}| \eta^\e(s )-\eta^\e(0)|^p \Big].
\ea
For any $\eta_0\in(0,1)$, $X_0(\cdot)\in C^2([0,1],\bR)$ we have
\ba
\E_{X_0,\eta_0}|\eta^\e(s)- \eta_0|^p
&\leq 
2^{p-1}\E_{X_0,\eta_0}\Big|\int^s_0 f(\eta^\e(r), X^\e(r,\eta^\e(r)),Z^\e (r))\,\di r\Big|^p \\
&+2^{p-1}\E_{X_0,\eta_0}\Big| \int^s_0 \sigma(\eta^\e(r), X^\e(r,\eta^\e(r)),Z^\e (r) )\,  \di W(r)\Big|^p\\
&\leq 
2^{p-1}s^{p-1} \E_{X_0,\eta_0}\int^s_0 \Big| f(\eta^\e(r), X^\e(r,\eta^\e(r)),Z^\e (r))\Big|^p \,\di r\\
&+2^{p-1}\Big(\frac{p(p-1)}{2}\Big)^{p/2} s^{(p-2)/2}\E_{X_0,\eta_0} \int^s_0 \Big|\sigma(\eta^\e(r), X^\e(r,\eta^\e(r)),Z^\e (r))\Big|^p\,  \di r\\
&\leq C_1 ( s^{p} + s^{p/2})\leq C_\eta(p)\Delta^{p/2}.
\ea
\end{proof}

For any $\e\in(0,1]$ and $\Delta=\Delta(\e)\in(0,1]$ denote
\ba
{}[s]_\Delta:=\Big\lfloor \frac{s}{\Delta} \Big\rfloor \Delta.
\ea
We define the following auxiliary processes:
\ba 
\label{diffeqXtildeetatilde}
\widetilde X^\e(t,x)
&=X^\e(k \Delta,x)+\frac{1}{\e} \int^t_{k \Delta} F(x, \eta^\e(k \Delta),\widetilde X^\e(s,x),\widetilde{Z}^\e(s))\,\di s\\
&+\frac{1}{\sqrt\e }\sum_{j=1}^\infty \int^t_{k \Delta} \Sigma^j(x,  \eta^\e(k \Delta),\widetilde X^\e(s,x),\widetilde{Z}^\e(s))\, \di B^j(s), 
\quad t\in[k \Delta, (k+1)\Delta),\quad k\geq 0,\\
\widetilde{Z}^\e(t)&=\int_0^1 \widetilde X^\e(t,x)\,\mu(\di x),\quad t\geq 0.
\ea
If $\sigma$ depends on $X$ and $Z$, i.e., $\sigma=\sigma(\eta,X,Z)$, for $t\geq 0$ we define the process
\ba
\widetilde  \eta^\e(t)&=\eta_0+\int^t_{0} f( \eta^\e([s]_\Delta),
\widetilde  X^\e(s ,\eta^\e([s]_\Delta)),\widetilde{Z}^\e(s))\,\di s 
+ \int^t_{0} \sigma ( \eta^\e([s]_\Delta),
\widetilde  X^\e(s ,\eta^\e([s]_\Delta)),\widetilde{Z}^\e(s))\,  \di W(s),\\
\ea
whereas for $\sigma=\sigma(\eta)$ we consider the process
\ba
\label{e:bar-eta}
\bar  \eta^\e(t)&=\eta_0+\int^t_{0} f( \eta^\e([s]_\Delta),
\widetilde  X^\e(s ,\eta^\e([s]_\Delta)),\widetilde{Z}^\e(s))\,\di s 
+ \int^t_{0} \sigma ( \eta^\e(s))\,  \di W(s).
\ea
The process $(\widetilde X^\e, \widetilde{Z}^\e)$ is obtained as a solution of the stochastic differential 
system \eqref{diffeq} with $f=\sigma=0$ on the intervals $[k \Delta, (k+1)\Delta]$
whereas the processes $\widetilde  \eta^\e$ and $\bar \eta^\e$ are obtained by integration.

Note that on each interval $[k\Delta,(k+1)\Delta]$
\ba
\label{e:in_law}
\text{Law}\Big(\widetilde X^\e(k\Delta+t,x),\ t\in[0,\Delta]\Big)
=\text{Law}\Big( \xi^{\eta}(t/\e,x),\ t\in [0,\Delta] \Big|\eta=\eta^\e(k\Delta), \xi^{\eta}(0,x)= X^\e(k \Delta,x)\Big).
\ea
We choose $\Delta=\Delta(\e)\in (0,1]$ such that the following holds true: for any $C>0$ and any $\gamma\in(0,1)$
\ba
\label{Deltabehavior}
&\Delta(\e)\to 0,\quad \e\to 0,\\
&\frac{\Delta(\e)}{\e} \to \infty,\quad \e\to 0,\\
&\Delta \Big(\frac{\Delta^2}{\e^2}+\frac{\Delta}{\e}\Big)
\ex^{C(\frac{\Delta^2}{\e^2}+\frac{\Delta}{\e})}= \mathcal O(\e^\gamma),\quad \e\in(0,1] .
\ea
\begin{rem}
For instance,
\ba
\Delta(\e):=\e \sqrt[4]{\ln\Big(1+\frac{1}{\e}\Big)}\wedge 1
\ea 
satisfies conditions \eqref{Deltabehavior}. 
\end{rem}
\begin{proof}
Indeed, let 
\ba
t^4=\ln\Big(1+\frac{1}{\e}\Big),\quad t\in[\ln 2,\infty).
\ea
Hence 
\ba
\e=\frac{1}{\ex^{t^4}-1}
\ea
and 
\ba
\Delta=\e \sqrt[4]{\ln\Big(1+\frac{1}{\e}\Big)}= \frac{t}{\ex^{t^4}-1}.
\ea
Therefore, for any $C>0$ and any $\gamma\in (0,1)$ there is $C_1>0$ such that the inequality
\ba
\Delta \Big(\frac{\Delta^2}{\e^2}+\frac{\Delta}{\e}\Big) \ex^{C(\frac{\Delta^2}{\e^2}+\frac{\Delta}{\e})}
= \frac{t^2(t+1)}{\ex^{t^4}-1}\ex^{C(t^2+t)}
\leq C_1 \Big(\frac{1}{\ex^{t^4}-1}\Big)^{\gamma}=C_1\e^\gamma
\ea
holds true for all $t\geq \ln 2$.
\end{proof}

\begin{lem}
There is a constant $C_X>0$ such that  for all $\e\in(0,1]$
\ba
\label{diffXXtilde}
\sup_{t\geq 0}  \sup_ {x\in [0,1]}\E_{X_0,\eta_0}
| X^\e(t,x)- \widetilde X^\e(t,x)|^2\leq C_X \sqrt\e. 
\ea

\end{lem}
\begin{proof}
Let $k\geq 0$ and $t\in [k\Delta,(k+1)\Delta)$. For each $x\in[0,1]$ we write 
\ba
\label{XXtildemc}
&\E_{X_0,\eta_0} | X^\e(t,x) -  \widetilde X^\e(t,x)|^2\\
& \leq \frac {2} { \e^2 }  \E_{X_0,\eta_0}  \Big|
\int_{k \Delta}^t\Big( F(x,\eta^\e(s),X^\e(s,x),Z^\e (s))- F(x,  \eta^\e(k \Delta),\widetilde X^\e(s,x),\widetilde Z^\e(s))\Big)\,\di s\Big|^2\\
&+\frac {2 }{ \e } \E_{X_0,\eta_0} 
\Big| \sum_{j=1}^\infty \int^t _{k \Delta} 
\Big(\Sigma^j(x,\eta^\e(s),X^\e(s,x),Z^\e (s))
- \Sigma^j(x, \eta^\e(k \Delta),\widetilde X^\e(s,x),\widetilde{Z}^\e(s))\Big)\, \di B^j(s) \Big|^2\\
&\leq 
\frac {2\Delta} { \e^2 }  \E_{X_0,\eta_0} \int_{k \Delta}^t 
 \Big|F(x,\eta^\e(s),X^\e(s,x),Z^\e (s))- F(x,  \eta^\e(k \Delta),\widetilde X^\e(s,x),\widetilde Z^\e(s))\Big|^2\,\di s\\
&+\frac {2}{ \e } \E_{X_0,\eta_0} 
\int^t _{k \Delta} 
\Big\|\Sigma(x,\eta^\e(s),X^\e(s,x),Z^\e (s))- \Sigma(x, \eta^\e(k \Delta),\widetilde X^\e(s,x),\widetilde{Z}^\e(s))\Big\|^2_{l^2}\, \di s\\
&\leq C_1\Big(\frac {\Delta} { \e^2 }+\frac{1}{\e}\Big)
\E_{X_0,\eta_0} \int_{k \Delta}^t \Big( |\eta^\e(s)- \eta^\e(k \Delta)|^2
+|X^\e(s,x)-\widetilde X^\e(s,x)|^2+|Z^\e (s)-\widetilde{Z}^\e(s)|^2\Big)\,\di s,
\ea
where $C_1>0$ depends on the Lipschitz constants of $F$ and $\|\Sigma\|_{l^2}$.
The first summand of the last integrand is estimated with the help of Lemma \ref{etaboundthm} with $p=2$ that gives us
\ba
\label{XXtildemceq1}
\sup_{t\in[k\Delta,(k+1)\Delta]}\E_{X_0,\eta_0}|\eta^\e(t)- \eta^\e(k \Delta)|^2 
\leq C_\eta(2) \Delta. 
\ea
Further, we estimate
\ba
\label{XXtildemceq2}
\E_{X_0,\eta_0} \int^{ t}_{k \Delta}|Z^\e (s)-\widetilde Z^\e(s)|^2\,\di s
&=\E_{X_0,\eta_0} \int^{t}_{k \Delta}\Big| \int^1_0  (X^\e (s,y)-    \widetilde  X^\e  (s,y))\,\mu (\di y) \Big|^2\,\di s\\
&\leq M\E_{X_0,\eta_0} \int^{t}_{k \Delta}\int^1_0 |X^\e (s,y)-    \widetilde  X^\e  (s,y)|^2\,|\mu| (\di y)\,\di s\\
&= M\int^{t}_{k \Delta}\int^1_0 \E_{X_0,\eta_0}  |X^\e (s,y)-    \widetilde  X^\e  (s,y)|^2\,|\mu| (\di y)\,\di s\\
&\leq M^2 \int_{k \Delta}^t \sup_ {x\in [0,1]} \E_{X_0,\eta_0}|  X^\e (s,x)- \widetilde  X^\e  (s,x) |^2\,\di s.
\ea
Combining  \eqref{XXtildemc}, \eqref{XXtildemceq1} and \eqref{XXtildemceq2}
we can find  a constant $C_2>0$ not depending on $\e$ and $\Delta$, such that
\ba
\sup_ {x\in [0,1]}&\E_{X_0,\eta_0} | X^\e(t,x) -  \widetilde X^\e(t,x)|^2\\
 &\leq  C_2 \Delta^2\Big(\frac {\Delta} {\e^2}+\frac{1}{\e}\Big)
 +  C_2\Big(\frac {\Delta}{\e^2}+ \frac {1}{\e}\Big) 
 \int_{k \Delta}^t \sup_ {x\in [0,1]}\E_{X_0,\eta_0} |X^\e(s,x)-\widetilde X^\e(s,x)|^2\,\di s.
\ea
With the help of Gronwall's lemma and \eqref{Deltabehavior} with $\gamma=1/2$, it follows that for $t\in[k\Delta,(k+1)\Delta]$ for some $C_3>0$ we have
\ba
\sup_ {x\in [0,1]}\E_{X_0,\eta_0}| X^\e(t,x) -  \widetilde X^\e(t,x)|^2
\leq C_2\Delta \Big(\frac{\Delta^2 }{ \e^2 }+ \frac { \Delta } { \e }\Big) 
\ex^{C_2(\frac{\Delta^2}{\e^2}+ \frac{\Delta}{\e})}\leq C_3 \sqrt\e.
\ea
Since all the constants do not depend on $k$, this estimate holds uniformly w.r.t.\ $t\geq 0$. 
\end{proof}

\begin{lem}\label{lemdiffetaetatilde}
For any $T>0$
\begin{align}
\label{diffetaetatilde}
&\E_{X_0,\eta_0}\sup_{t\in [0,T]}| \eta^\e(t)- \widetilde \eta^\e(t )|^2\to 0,\quad \e\to 0,\\
\label{e:bareta}
&\E_{X_0,\eta_0}\sup_{t\in [0,T]}| \eta^\e(t)- \bar \eta^\e(t )|^2\to 0,\quad \e\to 0. 
\end{align}
\end{lem}
\begin{proof}
We prove the limit \eqref{diffetaetatilde}. The limit \eqref{e:bareta} is obtained analogously.

Using the Lipschitz confinuity of $f$ and $\sigma$ and the Doob inequality we get
\ba
\label{etaetatildemc}
&\E_{X_0,\eta_0}\sup_ {t\in [0,T]} |\eta^\e(t)- \widetilde \eta^\e(t)|^2\\
&\leq 2\E_{X_0,\eta_0}\sup_ {t\in [0,T]} \Big| \int^t_0 f(\eta^\e(s),X^\e(s,\eta^\e(s)),Z^\e (s)  )
-f(\eta^\e([s]_\Delta), \widetilde  X^\e(s,\eta^\e([ s]_\Delta)),\widetilde{Z}^\e(s))\,\di s\Big|^2\\
&\quad+ 2\E_{X_0,\eta_0}\sup_ {t\in [0,T]} \Big| \int^t_0 \sigma(\eta^\e(s),X^\e(s,\eta^\e(s)),Z^\e (s)  )
-\sigma(\eta^\e([s]_\Delta), \widetilde  X^\e(s,\eta^\e([ s]_\Delta)),\widetilde{Z}^\e(s))\,\di W(s)\Big|^2\\
&\leq C_1 T  \E_{X_0,\eta_0} 
\int^T_0 \Big(|\eta^\e(s)-\eta^\e([s]_\Delta)|^2
+|X^\e(s,\eta^\e(s))-\widetilde  X^\e(s,\eta^\e([ s]_\Delta)) |^2
+|Z^\e (s)-\widetilde{Z}^\e(s)|^2\Big)\, \di s   \\
&\leq C_2 T \Big(\Big\lfloor \frac{T} \Delta\Big\rfloor+1\Big)\times\\
&\quad \times
\max_{0\leq k\leq \lfloor \frac{T} \Delta\rfloor}\E_{X_0,\eta_0}
\int^{(k+1)\Delta}_{k \Delta} \Big( |\eta^\e(s)-\eta^\e(k\Delta)|^2
+|X^\e(s,\eta^\e(s))-   X^\e(s ,\eta^\e(k\Delta))|^2\\
&\hspace{4.5cm}+|X^\e (s ,\eta^\e (k\Delta ) )-\widetilde  X^\e (s ,\eta^\e (k\Delta ) )|^2 +|Z^\e (s)-\widetilde{Z}^\e(s)|^2\Big)\, \di s  .
\ea
The first summand in the latter integrand is estimated by means of 
\eqref{XXtildemceq1}:
\ba
\E_{X_0,\eta_0}
\int^{(k+1)\Delta}_{k \Delta} |\eta^\e(s)-\eta^\e(k\Delta)|^2\, \di s   \leq C_\eta(2)\Delta^2.
\ea
The third summand is estimated by \eqref{diffXXtilde}: 
\ba
\int^{(k+1)\Delta}_{k \Delta}\E_{X_0,\eta_0}
&|X^\e (s ,\eta^\e (k\Delta ) )-\widetilde  X^\e (s ,\eta^\e (k\Delta ) )|^2\, \di s   \\ 
&\leq \Delta \sup_{s\in[k\Delta,(k+1)\Delta]}\E_{X_0,\eta_0}|X^\e (s ,\eta^\e (k\Delta ) )-\widetilde  X^\e (s ,\eta^\e (k\Delta ) )|^2\\ 
&\leq \Delta\sup_{s\in [0,\Delta ]}\E_{X_0,\eta_0}\Big[\sup_{x\in[0,1]}
\E_{\eta^\e(k\Delta), X^\e (k\Delta,\cdot)}|X^\e (s,x)-\widetilde  X^\e (s,x)|^2\Big]\\
&\leq C_X \Delta\sqrt\e .
\ea
The fourth summand is estimated with the help of \eqref{XXtildemceq2} and \eqref{diffXXtilde}:
\ba
\int^{(k+1)\Delta}_{k \Delta}\E_{X_0,\eta_0}
|Z^\e (s)-\widetilde Z^\e (s)|^2\, \di s  
&\leq M^2 \int_{k \Delta}^{(k+1)\Delta} \sup_ {x\in [0,1]} \E_{X_0,\eta_0}|  X^\e (s,x)- \widetilde  X^\e  (s,x) |^2\,\di s\\
&\leq M^2C_X  \Delta\sqrt\e .
\ea

It remains to estimate the second summand, where we use the Sobolev embedding theorem, H\"older's inequality, Theorem \ref{wnormmomentsboundthm} and Theorem \ref{etaboundthm}. We choose $2<p<p^*$ such that conditions of Theorem \ref{wnormmomentsboundthm} are fulfilled. We get 
\ba
\E_{X_0,\eta_0} &\int^{(k+1)\Delta}_{k \Delta} |X^\e(s,\eta^\e(s))-   X^\e(s ,\eta^\e(k\Delta))|^2\, \di s  \\
&\leq \Delta \sup_ {s\in [k\Delta,(k+1)\Delta]}\E_{X_0,\eta_0} |X^\e(s,\eta^\e(s))-   X^\e (s ,\eta^\e (k\Delta ))|^2 \\
&=\Delta \sup_ {s\in [k\Delta,(k+1)\Delta]}\E_{X_0,\eta_0} |\nabla X^\e(s,\theta^{\e,k}(s))|^2\cdot |\eta^\e(s)-\eta^\e(k\Delta)|^2 \\
&\leq C^2_\text{Sob}(1,p)\Delta \sup_ {s\in [k\Delta,(k+1)\Delta]} 
\E_{X_0,\eta_0} \| \nabla X^\e(s,\cdot)\|^2_{W^{1,p}} \cdot |\eta^\e(s)-\eta^\e(k\Delta)|^2 \\ 
&\leq C_\text{Sob}^2(1,p) \Delta \sup_ {s\in [k\Delta,(k+1)\Delta]} 
\Big(\E_{X_0,\eta_0}\| \nabla X^\e(s,\cdot)\|_{W^{1,p}}^{ p}\Big)^{2/p}
\Big(\E_{X_0,\eta_0}|\eta^\e(s)-\eta^\e (k\Delta)|^{ \frac{2p}{p-2}}\Big)^{\frac{p-2}{p}}\\
&\leq C_3\Delta^2   \Big(1+\|\nabla X_0(\cdot)\|^{p}_{W^{1,p}} + \| \nabla X_0(\cdot)\|^{2p^*}_{L^{2p^*}} \Big) ^ {2/p} ,
\ea
where $\theta^{\e,k}(s)\in (0,1)$

Altogether we find a constant $C>0$ such that
\ban
 \E_{X_0,\eta_0}\sup_ {t\in [0,T]} | \eta^\e(t )- \widetilde \eta^\e(t )|^2
 &\leq C\Big(\sqrt\e + \Delta(\e)  \Big(1+\|\nabla X_0(\cdot)\|^{p}_{W^{1,p}} + \| \nabla X_0(\cdot)\|^{2p^*}_{L^{2p^*}} \Big)   \Big) .
\ean
\end{proof}

\subsection{Weak convergence. Proof of Theorem \ref {averagingwc}}

The proof of the weak convergence of $\eta^\e$ to $\widehat\eta$ consists, as usual, in the proof of
the weak relative compactness of the family $\{\eta^\e\}_{\e\in(0,1]}$ and the proof of the convergence of finite dimensional
distributions.

1. Proof of convergence of final dimensional distributions. 
Let 
\ba
 \widehat L \varphi( x )= \frac{\widehat\sigma^2(x)}{2} \varphi'' (x) +  \widehat f(x) \varphi'(x),\quad \varphi\in C^3 ([0,1],\bR),
\ea
be the generator of the limit diffusion $\widehat \eta$. By Theorem 8.10 in Chapter 4 in Ethier and Kurtz \cite{EthierK-86},
we show that for any $m\geq 1$,
any bounded measurable function $\Phi=\Phi(x_1,\dots,x_m)$, any $\varphi\in C^2_b([0,1],\bR)$ 
and any $0 < t_1 < \dots< t_m< s \leq  s+t \leq  T$ it holds
\ba
\E_{X_0,\eta_0} \Phi(\eta^\e(t_1),\ldots,\eta^\e(t_m))\Big(\varphi(  \eta^\e(t))-\varphi(\eta^\e(s))
-\int^t_s \widehat L \varphi(  \eta^\e(u))\,\di u\Big)\to 0,\quad \e\to 0,
\ea
Recall the processes $\widetilde X^\e$, $\widetilde Z^\e$, and $\widetilde \eta^\e$ defined in \eqref{diffeqXtildeetatilde} and write
\ba
\label{e:3terms}
\E_{X_0,\eta_0}  &\Phi(\eta^\e(t_1),\ldots,\eta^\e(t_m))\Big(\varphi(  \eta^\e(s+t))-\varphi(\eta^\e(s))-\int^{s+t}_s \widehat L \varphi(  \eta^\e(u))\,\di u\Big)\\
&=\E_{X_0,\eta_0}  \Phi (\eta^\e(t_1),\ldots,\eta^\e(t_m) )\Big(\varphi(  \eta^\e(s+t))-\varphi(\widetilde\eta^\e(s+t))\Big)\\
&-\E_{X_0,\eta_0}  \Phi (\eta^\e(t_1),\ldots,\eta^\e(t_m) )\Big(\varphi(\eta^\e(s))-\varphi(\widetilde\eta^\e(s))\Big)\\
&+\E_{X_0,\eta_0}  \Phi (\eta^\e(t_1),\ldots,\eta^\e(t_m) )\Big(\varphi( \widetilde \eta^\e(s+t))-\varphi(\widetilde\eta^\e(s))-  \int^{s+t}_s \widehat L \varphi(  \eta^\e(u))\,\di u\Big)\\
&=:I_1^\e+I_2^\e+I_3^\e.
\ea
For $0\leq s+t\leq T$ using the boundedness of $\Phi$ and the Lipschitz property of $\varphi$ we estimate:
\ba
&\Big|\E_{X_0,\eta_0}  \Phi(\eta^\e(t_1),\ldots,\eta^\e(t_m) ) \Big(\varphi(  \eta^\e(s+t))-\varphi(\widetilde\eta^\e(s+t))\Big)\Big|^2
&\leq C_1 \E_{X_0,\eta_0} \Big| \eta^\e(s+t)-\widetilde\eta^\e(s+t)\Big|^2.
\ea
Hence, the summands $I_1^\e$ and $I_2^\e$ converge to zero as $\e\to 0$ due to Lemma 
\ref {lemdiffetaetatilde}.

%

Therefore we have to estimate   $I_3^\e$.
By the It\^o formula, we get
\ba
\label{e:34}
\varphi (\widetilde\eta^\e(s+t))- \varphi (\widetilde\eta^\e(s))
=& \int_s^{s+t} \varphi '(\widetilde  \eta^\e(u)) f\Big( \eta^\e ([u]_\Delta,\widetilde  X^\e(u ,\eta^\e([u]_\Delta)), \widetilde  Z^\e (u)\Big)\,\di u \\
&+\int_s^{s+t} \varphi '(\widetilde  \eta^\e(u)) 
\sigma\Big( \eta^\e ([u]_\Delta,\widetilde  X^\e(u ,\eta^\e([u]_\Delta)), \widetilde  Z^\e (u)\Big)\,\di W(u)\\
&+ \frac12  \int_s^{s+t}
\varphi ''(\widetilde  \eta^\e(u)) \sigma\Big( \eta^\e ([u]_\Delta,\widetilde  X^\e(u ,\eta^\e([u]_\Delta)), \widetilde  Z^\e (u)\Big)^2\,\di u.
\ea
Let $\Delta>0$ be small enough such that $0<s-t_m<\Delta$.
Taking conditional expectation we see that terms containing stochastic integrals vanish. Let us consider the terms containing the functions $f$ and
$\widehat f$. We have
\ba
\E_{X_0,\eta_0} & \Phi (\eta^\e(t_1),\ldots,\eta^\e(t_m) )
\int_s^{s+t}
\Big[ \varphi '(\widetilde  \eta^\e(u)) f\big( \eta^\e ([u]_\Delta,\widetilde  X^\e(u ,\eta^\e([u]_\Delta)), \widetilde  Z^\e (u)\big) -       \varphi'(\eta^\e(u))\widehat f(\eta^\e(u))\Big]\,\di u \\
&=\E_{X_0,\eta_0} \Phi (\eta^\e(t_1),\ldots,\eta^\e(t_m) )\Big(
\sum_{k=[\frac{s}{\Delta}]+1 }^{[\frac{s+t}{\Delta}]} \int_{k\Delta}^{(k+1)\Delta} 
+\int_{s}^{[s]_\Delta+\Delta}
+\int_{[s+t]_\Delta}^{s+t}\Big)
\Big[ \cdots\Big]\,\di u \\
&=S_1+S_2+S_3.
\ea
It is clear that for some $C_2>0$,
\ba
|S_2|+|S_3|\leq C_2\Delta.
\ea
For each $[\frac{s}{\Delta}]+1 \leq k\leq [\frac{s+t}{\Delta}]$ we write
\ba
\E_{X_0,\eta_0} &   \Phi (\cdots )
\int_{k\Delta}^{(k+1)\Delta} 
\Big( \varphi '(\widetilde  \eta^\e(u)) f\big( \eta^\e ([u]_\Delta,\widetilde  X^\e(u ,\eta^\e([u]_\Delta)), \widetilde  Z^\e (u)\big) -       \varphi'(\eta^\e(u))\widehat f(\eta^\e(u))\Big)
\,\di u\\
&=\E_{X_0,\eta_0} \Phi (\cdots )
\int_{k\Delta}^{(k+1)\Delta} 
\Big(\varphi '(\widetilde  \eta^\e(u))- \varphi'(\eta^\e(k\Delta))\Big)   f\big( \eta^\e (k\Delta),\widetilde  X^\e(u ,\eta^\e(k \Delta)), \widetilde  Z^\e (u)\big) \,\di u\\
&+\E_{X_0,\eta_0}  \Phi (\cdots ) 
\int_{k\Delta}^{(k+1)\Delta} \Big( \varphi'(\eta^\e(k\Delta)) \widehat f( \eta^\e (k\Delta)) -  \varphi'(\eta^\e(u)\widehat f(\eta^\e(u))\Big)\,\di u\\
&+\E_{X_0,\eta_0} \Phi (\cdots )
\int_{k\Delta}^{(k+1)\Delta} \Big( 
\varphi'(\eta^\e(k\Delta))f\big( \eta^\e (k\Delta),\widetilde  X^\e(u ,\eta^\e(k \Delta)), \widetilde  Z^\e (u)\big)
- \varphi'(\eta^\e(k\Delta)\widehat f(\eta^\e(k\Delta)) \Big)
\,\di u\\
&=S_{11,k}+S_{12,k}+S_{13,k}
\ea
With the help of Lemma~\ref{lemdiffetaetatilde} we estimate
\ba
|S_{11,k}|=\Big|\E_{X_0,\eta_0} & \Phi (\eta^\e(t_1),\ldots,\eta^\e(t_m) )
\int_{k\Delta}^{(k+1)\Delta} 
\Big(\varphi '(\widetilde  \eta^\e(u))- \varphi'(\eta^\e(u))\Big)   f\big( \eta^\e (k\Delta,\widetilde  X^\e(u ,\eta^\e(k \Delta)), \widetilde  Z^\e (u)\big) \,\di u \Big|\\
&\leq \Delta\cdot \|\Phi\|_\infty\cdot \|f\|_\infty \E_{X_0,\eta_0}
\sup_{u\in[s,T]}\Big|\varphi '(\widetilde  \eta^\e(u))- \varphi'(\eta^\e(u))\Big|  \\
&\leq 
\Delta\cdot \|\Phi\|_\infty\cdot \|f\|_\infty \cdot \|\varphi''\|_\infty \cdot \E_{X_0,\eta_0}
\sup_{u\in[0,T]}| \widetilde  \eta^\e(u) - \eta^\e(u)|=o(\Delta).
\ea
With the help of Lemma~\ref{etaboundthm} we obtain
\ba
|S_{12,k}|&\leq \Delta\cdot \|\Phi\|_\infty\cdot  \|(\varphi'f)'\|_\infty\cdot \E_{X_0,\eta_0}
\E_{\eta^\e(k\Delta), X^\e(k\Delta,\cdot)} |\eta^\e(u) - \eta^\e(k\Delta)|
=o(\Delta).
\ea
Finally, with the help of \eqref{e:in_law}, Assumption $\mathbf A_\text{w}$ and Theorem~\ref{wnormmomentsboundthm} we get
\ba
|S_{13,k}|&\leq \Delta\cdot \|\Phi\|_\infty\cdot \E_{X_0,\eta_0}
\E_{X^\e(k\Delta,\cdot),\eta^\e(k\Delta)}\Big| \frac{\e}{\Delta}\int_0^{\Delta/\e} 
\Big( \varphi'(\eta)f\big( \eta,\xi^{\eta}(u,\eta), \zeta^{\eta}(u)\big)\,\di u
- \varphi'(\eta)\widehat f(\eta\Big) \Big|_{\eta=\eta^\e (k\Delta)}\\
&\leq C_3\Delta \cdot \gamma\Big(\frac{\Delta}{\e}\Big) \cdot \E_{X_0,\eta_0}  \Big(1+\|X(k\Delta,\cdot)\|^2_{W^{1,2}}\Big)\\
&\leq C_4\Delta \cdot \gamma\Big(\frac{\Delta}{\e}\Big) \cdot  \Big(1+\|X_0(\cdot)\|^2_{W^{1,2}}\Big).
\ea
The same estimates hold true for the terms in \eqref{e:34} containing $\varphi''\sigma^2$.
    
Eventually, combining all the above estimates we get that
\ba
\Big|\E_{X_0,\eta_0} &\Phi(\eta^\e(t_1),\ldots,\eta^\e(t_m))\Big(\varphi(  \eta^\e(s+t))-\varphi(\eta^\e(s))-\int^{s+t}_s \widehat L \varphi(  \eta^\e(u))\,\di u\Big)\Big|\to 0,\quad \e\to 0,
\ea
which establishes the convergence of finite dimensional distributions of $\eta^\e$.

\noindent
2. In order to show tightness of the family $\{\eta^\e(\cdot)\}$ we prove that for any $T>0$ there is $C>0$ such that
\ba
\sup_{0\leq s<t\leq T}  \E_{X_0,\eta_0}|\eta^\e(t)- \eta^\e(s)|^4 \leq C|t-s|^2,
\ea
see Proposition 10.3 in Chapter 3 in Ethier and Kurtz \cite{EthierK-86}.
This follows from the following straightforward calculation that uses the moment inequality for stochastic integrals, see, e.g., 
Mao~\cite[Theorem 7.1]{Mao2007}:
\ba
\E_{X_0,\eta_0}&|\eta^\e(t)- \eta^\e(s)|^4\\
&= \E_{X_0,\eta_0}
\Big|\int^t_s  f(\eta^\e(u), X^\e(u,\eta^\e(u)),Z^\e (u))\,\di u +\int^t_s  \sigma(\eta^\e(u), X^\e(u,\eta^\e(u)),Z^\e (u))\,\di W(s)\Big|^4\\
&\leq
8 \E_{X_0,\eta_0}
\Big|\int^t_s  f(\eta^\e(u), X^\e(u,\eta^\e(u)),Z^\e (u))\,\di u\Big|^4\\
&\qquad +8 \E_{X_0,\eta_0}
\Big| \int^t_s  \sigma(\eta^\e(u), X^\e(u,\eta^\e(u)),Z^\e (u))\,\di W(u)\Big|^4\\
& \leq 8  \E_{X_0,\eta_0}|t-s|^3 \cdot \int^t_s | f(\eta^\e(s), X^\e(s,\eta^\e(s)),Z^\e (s))|^4\,\di s \\
&\qquad  +18 |t-s| \cdot \E_{X_0,\eta_0}\int^t_s  \sigma(\eta^\e(s), X^\e(s,\eta^\e(s)),Z^\e (s))^4\,\di s\\
&\leq 8\|f\|^4_\infty \cdot  |t-s|^4 +18\|\sigma\|^4_\infty  \cdot |t-s|^2 \\
&\leq C|t-s|^2. 
\ea

\subsection{Proof of Theorem \ref{averagingsc}}

In this Section we prove the convergence $\eta^\e \to\widehat \eta$ in the u.c.p.\ sense under the assumption that $\sigma(\eta,X,Z)=\sigma(\eta)$.
We continue using the auxiliary processes $\widetilde X^\e$ and $\widetilde Z^\e$ defined in 
\eqref{diffeqXtildeetatilde} and consider the  auxiliary processes $\bar\eta^\e$ defined as in \eqref{e:bar-eta}.
Note that the process $\bar\eta^\e$ differs from $\widetilde\eta^\e$ in the stochastic integral term.

\begin{lem} 
\label{l:ff}
For any $T>0$, there are $C_1$, $C_2>0$ such that
\ba
\label{e:ff}
\E_{X_0,\eta_0} \sup_ {t\in [0,T]} \Big| \int_0^t f( \eta^\e([s]_\Delta) ,
\widetilde X^\e(s ,\eta^\e([s]_\Delta)),\widetilde{Z}^\e(s))
-\widehat f ( \eta^\e(s))\,\di s\Big|^2 \leq C_1\sqrt \Delta + C_2\gamma\Big(\frac{\Delta}{\e}\Big)\Big(1+\|X_0\|_{W^{1,2}}\Big).
\ea
\end{lem}
\begin{proof}
By the triangle inequality we obtain
\ba
\E_{X_0,\eta_0} &\sup_ {t\in [0,T]} \Big| \int_0^t \Big( f( \eta^\e([s]_\Delta) ,
\widetilde X^\e(s ,\eta^\e([s]_\Delta)),\widetilde{Z}^\e(s))
-\widehat f ( \eta^\e(s))\Big) \,\di s\Big|\\
&\leq 
\E_{X_0,\eta_0} \sup_ {t\in [0,T]} \Big| \int_0^t\Big( \widehat f ( \eta^\e([s]_\Delta))
-\widehat f ( \eta^\e(s))\Big)\,\di s\Big|\\
&+\E_{X_0,\eta_0} \sup_ {t\in [0,T]} \Big| \int_0^t \Big(f( \eta^\e([s]_\Delta) ,
\widetilde X^\e(s ,\eta^\e([s]_\Delta)),\widetilde{Z}^\e(s))
-\widehat f ( \eta^\e([s]_\Delta))\Big)\,\di s\Big|\\
&:= I_1^\e+I_2^\e.
\ea
First we make the following estimate with the help of Theorem~\ref{etaboundthm}:
\ba
\label{e:ff1}
|I_1^\e|^2&\leq \text{Lip}(\widehat f)^2\cdot  T\cdot \int_0^T\E_{X_0,\eta_0} \Big|\eta^\e([s]_\Delta) -  \eta^\e(s) \Big|^2\,\di s\\
&= C_3
\sum_{j=0}^{\lfloor T/\Delta\rfloor-1}
\int_{j\Delta}^{(j+1)\Delta}\E_{X_0,\eta_0} \Big|\eta^\e(j\Delta) -  \eta^\e(s) \Big|^2\,\di s
+C_3
\int_{[T]_\Delta}^{T}
\E_{X_0,\eta_0} \Big|\eta^\e([T]_\Delta) -  \eta^\e(s) \Big|^2\,\di s\\
&\leq 
C_4 \Delta.
\ea
To estimate $I^\e_2$, we divide the integral into a sum of integrals, take conditional expectation and use 
the fact that the processes $\widetilde X^\e(s ,\eta^\e(k\Delta))$ and $\xi^{\eta^\e(k\Delta)}(s/\e)$ have the same law. Assumption $\mathbf A_\text{s}$ yields:  
\ba
 \E&_{X_0,\eta_0} \sup_{t\in[0, T]}
\Big|\int_0^t  f( \eta^\e([s]_\Delta) ,\widetilde X^\e(s ,\eta^\e([s]_\Delta)),\widetilde{Z}^\e(s))
-\widehat f ( \eta^\e([s]_\Delta)) \,\di s\Big|\\
&\leq \E_{X_0,\eta_0} \max_{0\leq k\leq T/\Delta} \sup_{t\in[0,\Delta]}
\Big|\int_0^{k\Delta}
\Big(f( \eta^\e([s]_\Delta) ,\widetilde X^\e(s ,\eta^\e([s]_\Delta)),\widetilde{Z}^\e(s))
-\widehat f ( \eta^\e([s]_\Delta))\Big)\,\di s\\
&\hspace{3.6cm} +\int_{k\Delta}^{k\Delta+t}
\Big(f( \eta^\e(k\Delta) ,\widetilde X^\e(s ,\eta^\e(k\Delta)),\widetilde{Z}^\e(s))
-\widehat f ( \eta^\e(k\Delta))\Big)\,\di s
\Big|\\
&\leq (\|f\|_\infty+\|\widehat f\|_\infty)\Delta\\
&+\E_{X_0,\eta_0}  \max_{0\leq k\leq T/\Delta} 
\Big|\int_0^{k\Delta}
\Big(f( \eta^\e([s]_\Delta) ,\widetilde X^\e(s ,\eta^\e([s]_\Delta)),\widetilde{Z}^\e(s))
-\widehat f ( \eta^\e([s]_\Delta))\Big)\,\di s \Big|\\
& \leq C_5\Delta + I_3^\e. 
\ea
We have
\ba
\label{e:fff}
I^\e_3
&
\leq 
\E_{X_0,\eta_0}   \max_{0\leq k\leq T/\Delta} 
\Big|\sum_{j=0}^k \int_{j\Delta}^{(j+1)\Delta}
\Big(f( \eta^\e(j\Delta) ,\widetilde X^\e(s ,\eta^\e(j\Delta)),\widetilde{Z}^\e(s))
-\widehat f ( \eta^\e(j\Delta))\Big)\,\di s\Big|\\
&\leq 
\sum_{0\leq k\leq T/\Delta} \E_{X_0,\eta_0}  
\Big|\int_{k\Delta}^{(k+1)\Delta}
\Big(f( \eta^\e(k\Delta) ,\widetilde X^\e(s ,\eta^\e(k\Delta)),\widetilde{Z}^\e(s))
-\widehat f ( \eta^\e(k\Delta))\Big)\,\di s\Big|\\
%
&\leq  
\sum_{0\leq k\leq T/\Delta} \E_{X_0,\eta_0}\E_{X^\e(k\Delta,\cdot),\eta^\e(k\Delta)} \Big|  
\int_{0}^{\Delta}
\Big(f( \eta ,\widetilde X^\e(s ,\eta,\widetilde{Z}^\e(s))
-\widehat f ( \eta)\Big)\,\di s\Big|\Big|_{\eta=\eta^\e(k\Delta)} \\
&\leq 
T\E_{X_0,\eta_0}\sup_{\eta\in[0,1]}
\E_{X^\e(k\Delta,\cdot ),\eta} \Big| \frac{1}{\Delta}
\int_{0}^{\Delta}
\Big(f( \eta ,\xi^\eta(s/\e),\zeta^\eta(s/\e))
-\widehat f ( \eta)\Big)\,\di s\Big| \\
 &\leq
T \E_{X_0,\eta_0}\sup_{\eta\in[0,1]}\E_{X^\e(k\Delta,\cdot),\eta} \Big| 
\frac{\e}{\Delta}\int_{0}^{\Delta/\e}
\Big(f( \eta ,\xi^\eta(s ),\zeta^\eta(s ))
-\widehat f ( \eta)\Big)\,\di s\Big| \\
&\leq 
C_6 T \gamma\Big(\frac{\Delta}{\e}\Big)\cdot\E_{X_0,\eta_0}\Big(1+ \|X^\e(k\Delta,\cdot)\|_{W^{1,2}} \Big)\\
&\leq 
C_7 T \gamma\Big(\frac{\Delta}{\e}\Big)\cdot  \Big(1+ \|X_0(\cdot)\|_{W^{1,2}}\Big).
\ea
Eventually for some $C_8>0$ we estimate
\ba
\label{e:ff2}
\E_{X_0,\eta_0} \sup_ {t\in [0,T]} \Big| & \int_0^t f( \eta^\e([s]_\Delta) ,
\widetilde X^\e(s ,\eta^\e([s]_\Delta)),\widetilde{Z}^\e(s))
-\widehat f ( \eta^\e(s))\,\di s\Big|^2 \\
&\leq C_8\Big( \sqrt\Delta +  \Delta + \gamma\Big(\frac{\Delta}{\e}\Big)\cdot  (1+ \|X_0(\cdot)\|_{W^{1,2}} )    \Big),
\ea
and the result is obtained.
\end{proof}

Now we are ready to prove the main result.
\begin{proof}[Proof of Theorem \ref{averagingsc}]
For $t\in [0,T]$,
consider the difference
\ba
\bar\eta^\e(t)-\widehat\eta(t)&
=\int_0^t \Big(f( \eta^\e([s]_\Delta),\widetilde X^\e (s ,\eta^\e([s]_\Delta)),\widetilde{Z}^\e(s))-\widehat f( \eta^\e(s) )\Big)\,\di s\\
&+\int^t_0 \Big(\widehat f(\eta^\e(s) )-\widehat f( \widehat \eta (s) )\Big)\,\di s 
+ \int^t_0\Big(\sigma( \eta^\e(s)) -\sigma(\widehat\eta (s) )\Big)\,  \di W(s).
\ea
and estimate
\ba
\label{e:aa}
\E_{X_0,\eta_0}&\sup_{s\in [0,t]}|\bar \eta ^\e(s)-\widehat\eta(s)|^2 \\
&\leq 
3\E_{X_0,\eta_0}\sup_{s\in [0,t]}
\Big|\int_0^s  \Big(f( \eta^\e([u]_\Delta),\widetilde X^\e (u ,\eta^\e([u]_\Delta)),\widetilde{Z}^\e(u))-\widehat f( \eta^\e(u) )\Big)\,\di u   \Big|^2\\
&+ 3\E_{X_0,\eta_0}\sup_{s\in [0,t]} \Big|\int^s_0 \Big(\widehat f(\eta^\e(u) )-\widehat f( \widehat \eta (u) )\Big)\,\di u\Big|^2\\ 
&+ 3\E_{X_0,\eta_0}\sup_{s\in [0,t]} \Big|\int^s_0\Big(\sigma( \eta^\e(u)) -\sigma(\widehat\eta (u) )\Big)\,  \di W(u)\Big|^2.
\ea
The first summand in \eqref{e:aa} is bounded from above by Lemma~\ref{l:ff}. Furthermore we have 
\ba
\E_{X_0,\eta_0}\sup_{s\in [0,t]} \Big|\int^s_0 \Big(\widehat f(\eta^\e(u) )-\widehat f( \widehat \eta (u) )\Big)\,\di u\Big|^2
&\leq T \E_{X_0,\eta_0} \int^t_0 \Big|\widehat f(\eta^\e(u) )-\widehat f( \widehat \eta (u) )\Big|^2\,\di u\\
&\leq T \text{Lip}(\widehat f)^2\E_{X_0,\eta_0} \int^t_0 |\eta^\e(u) - \widehat \eta (u)|^2\,\di u\\
&\leq C_1 \int^t_0 \E_{X_0,\eta_0}\sup_{r\in[0,u] } |\eta^\e(r) - \widehat \eta (r)|^2\,\di u.
\ea
Analogously, the Doob ineqiality (see \cite[Theorem 7.2]{Mao2007}) yields
\ba
\E_{X_0,\eta_0}\sup_{s\in [0,t]} \Big|\int^s_0\Big(\sigma( \eta^\e(u)) -\sigma(\widehat\eta (u) )\Big)\,  \di W(u)\Big|^2
&\leq 4 \E_{X_0,\eta_0} \int^t_0\Big|\sigma( \eta^\e(u)) -\sigma(\widehat\eta (u) )\Big|^2\,  \di u\\
&\leq C_2 \int^t_0 \E_{X_0,\eta_0}\sup_{r\in[0,u] } |\eta^\e(r) - \widehat \eta (r)|^2\,\di u.
\ea
Eventually we get
\ba
&\E_{X_0,\eta_0}\sup_{s\in [0,t]}|\eta^\e(s)-\widehat\eta(s)|^2   \\
&\leq
2\E_{X_0,\eta_0}\sup_{s\in [0,t]}|\eta^\e(s)-\bar\eta(s)|^2 
+2\E_{X_0,\eta_0}\sup_{s\in [0,t]}|\bar\eta ^\e(s)-\widehat\eta(s)|^2 \\
&\leq C_3 \int^t_0 \E_{X_0,\eta_0}\sup_{r\in[0,u] } |\eta^\e(r) - \widehat \eta (r)|^2\,\di u + R_T(\e),
\ea
where
$R_T(\e)\to 0$ as $\e\to 0$ by Lemmas~\ref{lemdiffetaetatilde} and \ref{l:ff}.
The application of the Gronwall inequality finishes the proof.
\end{proof}

%

\end{document}